\documentclass{article} 
\pdfoutput = 1
\usepackage[margin=1.25in,left=1.25in,right=1.25in]{geometry}
\RequirePackage[english]{babel}
\RequirePackage[pdftex]{graphicx}
\RequirePackage{threeparttable}
\RequirePackage[T1]{fontenc}
\RequirePackage{calc}
\RequirePackage{amsmath,amsthm,amsfonts,amssymb}
\RequirePackage{graphics,fancyhdr}
\RequirePackage{authblk}
\RequirePackage[OT1]{fontenc}
\RequirePackage{natbib}

\RequirePackage{lscape}
\RequirePackage{threeparttable}

\RequirePackage{lmodern}
\RequirePackage{microtype}
\RequirePackage{setspace}
\RequirePackage{caption,booktabs}
\RequirePackage{float}

\RequirePackage{epstopdf}           
\RequirePackage{titlesec}

\RequirePackage[colorlinks=true,citecolor=blue,urlcolor=blue]{hyperref}

\let\oldsection\section
\renewcommand{\section}{%
    \setcounter{equation}{0}%
    \oldsection%
}

\newcommand{\be}{\begin{eqnarray}}
\newcommand{\ee}{\end{eqnarray}}
\newcommand{\bSubE}{\begin{subequations}}
\newcommand{\eSubE}{\end{subequations}}
\newcommand{\bN}{\begin{enumerate}}
\newcommand{\eN}{\end{enumerate}}
\newcommand{\bNAlph}{\begin{enumerate}}
\newcommand{\eNAlph}{\end{enumerate}}
\newcommand{\bI}{\begin{itemize}}
\newcommand{\eI}{\end{itemize}}
\newcommand{\ba}{\begin{array}}
\newcommand{\ea}{\end{array}}

\newcommand{\R}{\mathbb{R}}

\newtheorem{theorem}{Theorem}
\newtheorem{proposition}{Proposition}
\newtheorem{corollary}{Corollary}
\newtheorem{lemma}{Lemma}

\theoremstyle{plain}

\theoremstyle{definition}
\newtheorem{definition}{Definition}
\newtheorem{example}{Example}

\newcommand{\promille}{%
    \relax\ifmmode\promillezeichen
          \else\leavevmode\(\mathsurround=0pt\promillezeichen\)\fi}
  \newcommand{\promillezeichen}{%
    \kern-.05em%
    \raise.5ex\hbox{\the\scriptfont0 0}%
    \kern-.15em/\kern-.15em%
    \lower.25ex\hbox{\the\scriptfont0 00}}

\usepackage{xcolor}
\begin{titlepage}\large
\title{A Generalized Representation of Faá di Bruno's Formula Using Multivariate Bell Polynomials}
\author{Michael P. Evers\thanks{Department of Economics, University of Hohenheim, Germany. E-mail: michael.evers@uni-hohenheim.de}\;  and Markus Kontny\thanks{Corresponding Author: Deutsche Bundesbank, Germany. E-mail: markus.kontny@bundesbank.de. The views expressed in this paper are not necessarily those of the
Deutsche Bundesbank or the Eurosystem.}}
\date{\today}

\end{titlepage}

\begin{document}
%
\maketitle
\begin{abstract} We provide a novel representation of the total n-th derivative of the multivariate composite function $f \circ g$, i.e. a generalized Faà di Bruno's formula. To this end, we make use of properties of the Kronecker product and the n-th derivative of the left-composite $f$, which allow the use of a multivariate form of partial Bell polynomials to represent the generalized Faà di Bruno's formula. We further show that standard recurrence relations that hold for the univariate partial Bell polynomial also hold for the multivariate partial Bell polynomial under a simple transformation. We apply this generalization of Faà di Bruno's formula to the computation of multivariate moments of the normal distribution.
\end{abstract}
%
\section{Introduction}
In this paper we set out to construct an easy to handle representation of higher total derivatives of composite functions which might be vector-valued and multivariate. We generalize the well-known Faà Di Bruno's formula in a way which is identical in structure to the univariate case. Specifically, we provide a parsimonious yet valuable representation for the n-th total derivative using a multivariate generalization of partial Bell polynomials:
\begin{equation}
d^{n} f(g(x)) \label{eq:start}
= \left( \sum_{k=1}^{n} f_{g^{k}}  \textbf{B}^g_{n,k}\right)  (dx)^{\otimes n},
\end{equation}
where $\textbf{B}^g_{n,k}$ is identical in structure to the univariate partial Bell polynomial over the respective matrices $\left(g_x, g_{x^2},\ldots,g_{x^{n-k+1}}\right)$ containing all partial derivatives of $g$ w.r.t $x$ at all orders $l=1,2,...,n-k+1$ of sizes $n_{y}\times n_{x}^{l}$, with scalar multiplication substituted by the Kronecker product $\otimes$ and with $ f_{g^{k}}$ being a $n_f \times n_y^k$ matrix collecting all k-th order partial derivatives of $f$ w.r.t $g$.\footnote{Representing all k-th order partial derivatives in one vector or matrix instead of using multi-indices or tensors is crucial for this representation. For example, such a representation has been advanced in \cite{holmquist_d-variate_1996} and \cite{magnus_concept_2010}.}

Faà Di Bruno's formula for the univariate case is a well-known result. An overview of its history can be found in \cite{johnson2002curious}. Multiple papers generalize Faà Di Bruno's formula for vector-valued and multivariate composite $f \circ g$ to obtain higher order partial derivatives. Examples for multivariate cases can be found in \cite{gzyl_multidimensional_1986}, \cite{constantine_multivariate_1996},  \cite{hardy_combinatorics_2006} and \cite{leipnik_multivariate_2007}. However, they do not employ Bell polynomials to express the combinatorial structure of the higher order composite derivatives. \cite{noschese2003differentiation} and \cite{schumann_multivariate_2019} make use of Bell polynomials to derive higher order partial derivatives and \cite{di2011new} applies classical umbra calculus. A generalization of Faà Di Bruno's formula to obtain rather the n-th total derivative are provided by \cite{mishkov_generalization_2000} and for the matrix of n-th order derivatives by \cite{chacon_higher_2021}. However, both representations to do not exploit the powerful properties of Bell polynomials.

In this paper, we offer a generalized representation of the n-th total derivative for a vector-valued and multivariate composite $f \circ g$ that takes advantage of both the properties of Bell polynomials as well as the Kronecker product. Firstly, we thereby systematically collect all higher order partial derivatives in a novel multivariate Bell Polynomial and make it thus directly applicable to compute e.g. Taylor series expansions. Secondly, the novel multivariate partial Bell polynomials inherit the common recurrence relationships familiar from the univariate case which leads us directly to apply them to a multivariate Faà Di Bruno's formula. Thirdly, we apply this novel representation to the computation of arbitrary moments of the multivariate normal distribution.

We start by recalling the univariate partial Bell polynomial and present a proof of the univariate version of Faà di Bruno's formula using induction in Section 2. This section serves as an expository and pedagogical guideline, since these results are already well-known in the univariate case, but our proof for the multivariate version works along similar lines. In Section 3, we introduce the notation for matrix-valued derivatives and present our multivariate version of the partial Bell polynomial. In this section, we also lay out the issues when attempting to use the multivariate partial Bell polynomial and how these can be resolved. In Section 4, we show, on the basis of commutation properties of the Kronecker product, that the familiar recurrence relations for partial Bell polynomials also hold in the multivariate case. We finally prove the generalized version of Faà Di Bruno's formula in Section 5. In Section 6, we apply the generalization to the computation of moments of the multivariate normal distribution. We add some concluding remarks in Section 7.
\\\\
\textbf{Notation and Definitions:}
Where it is not defined differently, functions $f(y)$ and $g(x)$ are vectors of size $n_f \times 1$ and $n_y \times 1$, respectively, given by $y = g\left(x\right)$ with $f: \R^{n_y} \rightarrow \R^{n_f}$ and $g: \R^{n_x} \rightarrow \R^{n_y}$\footnote{Obviously $f,g$ can be defined on more narrow domains as long as their are sufficiently often differentiable on these.}. Both $f$ and $g$ are at least n times differentiable. Further let $x=\left(x_1,\ldots,x_{n_x}\right)'$ and $y=\left(y_1,\ldots,y_{n_y}\right)'$ be column vectors. $\mathcal{M}_{r \times c}\left(\mathbb{R}\right)$ represents the set of all matrices over the real numbers with $r$ rows and $c$ columns. $I_n$ denotes the identity matrix of size $n \times n$. $Sym(m)$ is the symmetric group over the set of elements given by $\left\{ 1,\ldots m\right\}$ and $\sigma \in Sym(m)$ represents a permutation defined over these $m$ elements.

\section{Univariate Faà Di Bruno}
To arrive at some intuition and guidance for the proof of the multivariate Faà Di Bruno's formula, we start by providing a proof of the univariate case. We define the partial Bell polynomial\footnote{For reference see for example \cite{comtet1974advanced}.} over the function $g$ as follows
\begin{definition}[Univariate Partial Bell Polynomial]
For some $n,k \in \mathbb{N}$ with $n \geq k$ and a function $g:\R \rightarrow \R$ which is at least $n-k+1$ times differentiable on some domain $D \subset \R$, we can define the partial Bell polynomial by
\begin{align}\label{Bell-univariate}
\begin{array}{l}
B_{n,k}(g_{x}(x),g_{x^{2}}(x),\cdots,g_{x^{(n - k + 1)}}(x))\\[1.5ex]
= \sum {n! \over j_{1}!j_{2}!\cdots j_{n-k+1}!}\left({g_{x}(x) \over 1!}\right)^{j_{1}}\left({g_{x^{2}}(x) \over 2!}\right)^{j_{2}}\cdots \left({g_{x^{(n - k + 1)}}(x) \over (n-k+1)!}\right)^{j_{n-k+1}},
\end{array}
\end{align}
where $\left(g_x, g_{x^2},\ldots,g_{x^{n-k+1}}\right)$ denote the univariate partial derivatives of $g$ w.r.t $x$ at all orders $i=1,2,...,n-k+1$, and where the sum is taken over all sequences $j_{1}$, $j_{2}$, $j_{3}$,...., $j_{n-k+1}$ of non-negative integers such that these two conditions are satisfied:
\begin{align}\label{Bell-Conditions-uni}
\begin{array}{cc}
j_{1}+j_{2}+\cdots +j_{n-k+1}&=k\\[1.5ex]
j_{1}+2j_{2}+3j_{3}+\cdots +(n-k+1)j_{n-k+1}&=n.
\end{array}
\end{align}
To ease notation, we set
\begin{equation}
B^g_{n,k} := B_{n,k}(g_{x}(x),g_{x^{2}}(x),\cdots,g_{x^{(n - k + 1)}}(x)).
\end{equation}
Further it holds that $B^g_{n,k} = 0$ for $k>n$, $n = 0$, and $k=0$.
\end{definition}
Before we proceed we require the following Lemma:
\begin{lemma}[Differentiation Recurrence of Univariate Partial Bell Polynomials]
\label{lemma:bellrecur}
For any scalars $n, k$ and functions $g$ as defined above, it holds that
\begin{equation}\label{eq:recrel}
B^g_{n+1,k} = g_x(x) B^g_{n,k-1} + \frac{d}{d x} B^g_{n,k}.
\end{equation}
\end{lemma}
\begin{proof}
The result follows from two well known recurrence relations for the partial Bell polynomial\footnote{See for example \cite{charalambides2018enumerative}.}
\begin{align}
B_{n,k}^g &= \sum_{i = 1}^{n-k+1} {n-1 \choose i-1} g_{x^i} B_{n-i,k-1}^g \label{eq:recursion1}\\
\frac{d}{dx} B_{n,k}^g &= \sum_{i = 1}^{n-k+1} {n \choose i } g_{x^{i+1}} B_{n-i,k-1}^g \label{eq:recursion2}
\end{align}
Starting from the right-hand-side (RHS) of \eqref{eq:recrel} and using \eqref{eq:recursion2} we can see that
\begin{align}
 &\quad g_x B^g_{n,k-1} + \frac{d}{d x} B^g_{n,k} \\
 =&\quad  g_x  B^g_{n,k-1} + \sum_{i = 1}^{n-k+1} {n \choose i } g_{x^{i+1}} B_{n-i,k-1}^g\\
  =&\quad  g_x B^g_{n,k-1} + \sum_{i = 2}^{n-k+2} {n \choose i - 1} g_{x^{i}} B_{n-i+1,k-1}^g\\
    =&\quad  g_x B^g_{n,k-1} -  g_x B_{n,k-1}^g + \sum_{i = 1}^{n-k+2} {n \choose i - 1} g_{x^{i}} B_{n-i+1,k-1}^g \\
        =&\quad \sum_{i = 1}^{n-k+2} {n \choose i - 1} g_{x^{i}} B_{n-i+1,k-1}^g. \label{eq:recrel_last}
    \end{align}
By setting $n := n+1$ in equation \eqref{eq:recursion1} and comparing to \eqref{eq:recrel_last}, the proof is complete.
\end{proof}

Using Lemma 1 it is now straightforward to prove the univariate and scalar-valued case of Faà Di Bruno by induction:

\begin{theorem}[Univariate Faà Di Bruno's Formula Using Bell Polynomials]
For some $n,k \in \mathbb{N}$ with $n \geq k$ and function $g:\R \rightarrow \R$ which is at least $n-k+1$ differentiable on some domain $D_g \subset \R$ and a function $f:\R \rightarrow\R$ which is at least $n$ times differentiable on some domain $D_f \subset \R$, the $n$-th order composite derivative of $f \circ g$ is given by the Faà Di Bruno's formula:
\begin{align}\label{FaaDiBruno-Bell-univariate}
\frac{d^{n}}{d x^{n}} f(g(x))
= \sum_{k=1}^{n} f_{g^{k}}(g(x)) B^g_{n,k}.
\end{align}
\end{theorem}
\begin{proof} We prove the univariate case by induction. Start of induction for $n = 1$ gives
\begin{align}
\frac{d}{d x} f(g(x))
&= \sum_{k=1}^{1} f_{g^{k}}(g(x))  B_{n,k}^g\\
&= f_g\left(g(x)\right) B_{1,1}^g \\
& = f_g\left(g(x)\right) g_x(x),
\end{align}
which follows directly from the chain rule.
Now suppose \eqref{FaaDiBruno-Bell-univariate} holds for some arbitrary $n \in \mathbb{N}$.
Then we can write
\begin{align}
\frac{d^{n+1}}{d x^{n+1}} f(g(x))
&= \frac{d}{d x} \left[\frac{d^{n}}{d x^{n}} f(g(x))\right]\\
&= \frac{d}{d x} \left[\sum_{k=1}^{n} f_{g^{k}}(g(x)) B^g_{n,k}\right]\\
&= \sum_{k=1}^{n} f_{g^{k+1}}\left(g(x)\right) g_x(x) B^g_{n,k} + f_{g^k}\left(g(x)\right) \frac{d}{d x} B^g_{n,k}.
\end{align}
We can factor for $f_{g^{k}}\left(g(x)\right)$ after re-parameterizing the sums and noting that $B^g_{n,0}\left(\cdot\right) = B^g_{n,n+1}\left(\cdot\right) = 0$.
\begin{align}
&=\sum_{k=1}^{n+1} f_{g^{k}}\left(g(x)\right) g_x(x) B^g_{n,k-1} + \sum_{k=1}^{n+1} f_{g^k}\left(g(x)\right) \frac{d}{d x} B^g_{n,k} \\
&=\sum_{k=1}^{n+1} f_{g^{k}}\left(g(x)\right) \left( g_x(x) B^g_{n,k-1} + \frac{d}{d x} B^g_{n,k}\right).
\end{align}
Because
\begin{equation}
B^g_{n+1,k} = g_x(x) B^g_{n,k-1} + \frac{d}{d x} B^g_{n,k},
\end{equation}
as given by Lemma \ref{lemma:bellrecur}, we get
\begin{align}
\frac{d^{n+1}}{d x^{n+1}} f(g(x)) = \sum_{k=1}^{n+1} f_{g^{k}}\left(g(x)\right)B^g_{n+1,k}.
\end{align}
\end{proof}

\section{Multivariate Bell Polynomials: Definition and the Issue}


In this section, we define a multivariate and matrix-valued version of the partial Bell polynomial. To that purpose, we also quickly present the notation for vector- and matrix-valued representations of derivatives and differentials used in this paper. Let $F$ be a matrix of size $s \times t$. We define a differential operator ${\partial \over \partial x'}$ w.r.t the transpose of $x$ by
\begin{equation}
{\partial \over \partial x'} = \left({\partial \over \partial x_1},\ldots,{\partial \over \partial x_{n_x}} \right).
\end{equation}
First-order derivatives of matrices and vectors are then defined by
\begin{equation}
F_x = {\partial F \over \partial x'} := F \otimes {\partial  \over \partial x'}.
\end{equation}
Analogously, higher-order derivatives are defined by
\begin{equation}
 F_{x^k} = {\partial F \over \partial (x')^{\otimes k}} := F \otimes \left({\partial  \over \partial x'}\right)^{\otimes k}.
\end{equation}
This notation implies some structure on the ordering of partial derivative within ${\partial F \over \partial x'}$.\footnote{See \cite{tracy1993higher}, \cite{sultan1996moments}, and \cite{rogers1980matrix} who also make use of this notation.} More explicitly this structure corresponds to the replacement of each component $F_{i,j}$ of $F$ by its respective row vector of derivatives, i.e.
\begin{equation}
{\partial F_{i,j} \over \partial x'} = \left({\partial F_{i,j} \over \partial x_1},\ldots,{\partial F_{i,j} \over \partial x_d} \right).
\end{equation}

When F is a vector, denoted by $f$ in the following, we can write the corresponding first-order differential using
\begin{equation}
d f = f_x dx
\end{equation}
and the higher-order differentials by
\begin{equation}
d^k f = f_{x^k} \left(dx\right)^{\otimes k},\footnote{For more details on the the representation of vector- and matrix-valued derivatives and the corresponding rules of differentiation, see Appendix A.}
\end{equation}
where the Kronecker product power notation refers to  $\left(dx \right)^{\otimes k}= \underbrace{dx \otimes \cdots \otimes dx}_{k \; \text{times}}$.

We are now equipped to define our generalization of the partial Bell polynomial.

\begin{definition}[Multivariate and Matrix-Valued Partial Bell Polynomial]
Let $g_{x^{l}}(x)$ denote the $n_{y} \times n_{x}^{l}$ matrix of the $l$th-order partial derivatives of $g(\cdot)$, and where $\textbf{B}^g_{n,k}$ is a $n_{y}^k \times n_{x}^n$ matrix given by
\begin{align}\label{Bell-multivariate}
\begin{array}{l}
\textbf{B}^g_{n,k} := \textbf{B}_{n,k}(g_{x}(x),g_{x^{2}}(x),\cdots,g_{x^{(n - k + 1)}}(x)) =\\[1.5ex]
\sum {n! \over j_{1}!j_{2}!\cdots j_{n-k+1}!}\left({g_{x}(x) \over 1!}\right)^{\otimes j_{1}}\otimes\left({g_{x^{2}}(x) \over 2!}\right)^{\otimes j_{2}}\otimes\cdots \otimes \left({g_{x^{(n - k + 1)}}(x) \over (n-k+1)!}\right)^{\otimes j_{n-k+1}},
\end{array}
\end{align}
where the sum is taken over all sequences $j_{1}$, $j_{2}$, $j_{3}$,...., $j_{n-k+1}$ of non-negative integers such that these two conditions are satisfied:
\begin{align}\label{Bell-Conditions-uni}
\begin{array}{cc}
j_{1}+j_{2}+\cdots +j_{n-k+1}&=k\\[1.5ex]
j_{1}+2j_{2}+3j_{3}+\cdots +(n-k+1)j_{n-k+1}&=n.
\end{array}
\end{align}
\end{definition}
This definition is identical in structure to the univariate case with the exception that scalar multiplication is substituted by the Kronecker product. The bold writing of $\textbf{B}^g_{n,k}$ emphasizes that this Bell polynomial is a matrix object and not a scalar.

\paragraph{Non-Recurrence of Multivariate Partial Bell Polynomials - Example} It is important to stress that the recurrence relations of univariate Bell Polynomials do not simply carry over to the multivariate and matrix-valued case as the Kronecker products do not commute. To see this, suppose a multivariate and matrix-valued recurrence relation was given by
\begin{equation}
\label{eq:partialBellrecur_mult}
\textbf{B}_{n,k}^g = \sum_{i = 1}^{n-k+1} {n-1 \choose i-1} \left( \textbf{B}_{n-i,k-1}^g \otimes g_{x^i}\right),
\end{equation}
as in the univariate case. Consider for given $g$ the following partial Bell polynomials:
\begin{align}
\textbf{B}^g_{1,1} &= g_{x} \\
\textbf{B}^g_{2,1} &= g_{x^2} \\
\textbf{B}^g_{3,2} &= 3\left(g_x \otimes g_{x^2} \right).
\end{align}
If the recurrence was to hold true, it would require that
\begin{equation}
 \textbf{B}^g_{3,2}=  \textbf{B}^g_{2,1} \otimes g_{x} + 2 \left(\textbf{B}^g_{1,1} \otimes g_{x^2}\right).
\end{equation}

Simply plugging in the given partial Bell polynomials then leads to
\begin{equation}
3\left(g_{x} \otimes g_{x^2}\right) \neq g_{x^2} \otimes  g_{x} + 2\left(g_{x} \otimes g_{x^2}\right).
\end{equation}
The left-hand-side does not equal the right-hand-side because the Kronecker product does not commute.

\paragraph{Reestablishing Recurrence of Multivariate Partial Bell Polynomials - Example cont.}
However, it is well known that Kronecker products can be commuted with the help of so-called commutation matrices.
Therefore, given the existence of commutation matrices $L$ and $R$, we can write
\begin{equation}
 g_{x^2} \otimes  g_{x} + 2\left(g_{x} \otimes g_{x^2}\right) = L\left(g_x \otimes  g_{x^2}\right) R + 2\left(g_{x} \otimes g_{x^2}\right) .
\end{equation}
We will introduce commutation matrices in more detail in the next section.
Suppose now, that the function $g$ is the inner function of the composite function $f \circ g$, and that $dx$ is a differential of the arguments of $g$. Then it turns out, we will be able to write
\begin{align}
f_{g^2} \left( \textbf{B}^g_{2,1} \otimes g_{x} + 2 \left(\textbf{B}^g_{1,1} \otimes g_{x^2}\right) \right)  (dx)^{\otimes 3}
&= f_{g^2} \left( g_{x^2} \otimes  g_{x} + 2\left(g_{x} \otimes g_{x^2}\right) \right)   (dx)^{\otimes 3} \label{eq:step11} \\
&= f_{g^2} \left( L\left(g_x \otimes  g_{x^2}\right) R + 2\left(g_{x} \otimes g_{x^2}\right)\right)   (dx)^{\otimes 3} \\
&= f_{g^2} L\left(g_x \otimes  g_{x^2}\right) R   (dx)^{\otimes 3} +  f_{g^2} 2\left(g_{x} \otimes g_{x^2}\right)  (dx)^{\otimes 3} \label{eq:step21}\\
&= f_{g^2}  3\left(g_x \otimes  g_{x^2} \right) (dx)^{\otimes 3}, \label{eq:step22} \\
& = f_{g^2}  \textbf{B}^g_{3,2}  (dx)^{\otimes 3} \label{eq:step1n}
\end{align}
because $f_{g^2} L = f_{g^2}$ and $R (dx)^{\otimes 3} = (dx)^{\otimes 3}$.
$f_{g^2} L = f_{g^2}$ is true because the commutation matrix $L$ coincidentally only commutes the last two colored factors in $f_{g^2} = f \otimes \textcolor{blue}{\partial \over \partial x'}  \otimes \textcolor{red}{\partial \over \partial x'} $, which does not change the overall derivative. Similarly, the commutation matrix $R$ will commute the colored objects in $\textcolor{blue}{dx \otimes dx} \otimes \textcolor{red}{dx}$, which also doesn't change the differential.

In the following we will show the generalization of this insight for $n,k \geq 0$ in $\textbf{B}^g_{n,k}$. To that purpose, we start by generalizing the equality of \eqref{eq:step21} to \eqref{eq:step22}, after which we show how we are always able to write multivariate partial Bell polynomial recurrences as in \eqref{eq:step11} and \eqref{eq:step1n}.
%
%
%


\paragraph{Outline}
In the following we outline the steps to prove that we are able to generalize recurrence relations for multivariate partial Bell Polynomials.
Taking the example of $\textbf{B}_{3,2}$ as a  guideline, in the first step, \textbf{Step 1.}, we show that for all matrices $D_k,X_n$ with
\begin{align}
D_k & := D^{\otimes k} \\
X_n & := X^{\otimes n},
\end{align}
where $D$ is an arbitrary real-valued $1 \times n_y$ vector, and $X$ an arbitrary real-valued $n_x \times 1$ vector, we can write
\begin{equation}
D_k  \left(g_{x^2} \otimes  g_{x}\right) X_n = D_k \left(g_{x} \otimes g_{x^2}\right) X_n.
\end{equation}
However, we will prove this more generally for so-called partial base polynomial of the form
\begin{equation}
 g_{x}^{\otimes j_1} \otimes g_{x^2}^{\otimes j_2} \otimes \ldots \otimes g_{x^{n-k+1}}^{\otimes j_{n-k+1}}.
\end{equation}

We define partial Bell base polynomials and present relevant results for commutation matrices in the next section in detail. In principle, this will show that the ordering of the factors within a base polynomial does not matter.
In the second step, \textbf{Step 2.}, we then generalize the recurrence relation
\begin{equation}
D_k \textbf{B}^g_{3,2} X_{n+1} = D_k \left( \textbf{B}^g_{2,1} \otimes g_{x} + 2 \left(\textbf{B}^g_{1,1} \otimes g_{x^2}\right) \right) X_{n+1}.
\end{equation}
and show that all recurrence relationships in Lemma \ref{lemma:bellrecur} hold in the multivariate case. In the general case, this amounts to showing that
\begin{equation}
D_k \textbf{B}_{n,k}^g X_{n+1} = D_k \left[\sum_{i = 1}^{n-k+1} {n-1 \choose i-1} \left( \textbf{B}_{n-i,k-1}^g \otimes g_{x^i}\right) \right] X_{n+1}
\end{equation}
holds. It turns out that we can rewrite any univariate partial Bell polynomial recurrence into a multivariate partial Bell polynomial recurrence in that from.

By choosing $D_k = f_{g^k}$ and $X_n = \left(dx\right)^{\otimes n}$, Faà Di Bruno's formula for the multivariate case follows using a similar induction proof as in the univariate case.

\section{Recurrence of Multivariate Bell Polynomials}

In this section we present well-known results concerning the commutation matrices for Kronecker products, and prove the generalization of multivariate partial Bell polynomial recurrence relations.
\paragraph{Commutation Matrices}
In general, two matrices connected by a Kronecker product do not commute. However one can construct so-called commutation matrices, so that for a given matrix $A_1$ of size $r_1 \times c_1$ and matrix $A_2$ of size $r_2 \times c_2$ we get,
\begin{equation}
K_{r_2 r_1}\left(A_1 \otimes A_2\right) K_{c_1 c_2} = A_2 \otimes A_1.
\end{equation}
$K_{r_2 r_1}$ and $K_{c_1 c_2}$ are commutation matrices of sizes $r_2 r_1 \times r_2 r_1$ and $c_1 c_2 \times c_1 c_2$, respectively. It will also be useful to note some properties of commutation matrices:
\begin{enumerate}
\item $K_{mn} = K_{nm}$
\item $K_{mn} K_{mn}' = K_{mn} K_{nm} = I_{mn}$
\item $K_{1n} = K_{n1} = I_n$.
\end{enumerate}
For a more detailed introduction and an explicit representation of commutation matrices see \cite{magnus1979commutation} and \cite{magnus_matrix_1999}.

One can easily generalize the commutation result for matrices to arbitrarily many matrices for all cyclical permutations.\footnote{See \cite{magnus1979commutation}, Theorem 3.1 (xi).}
However, it is also possible to generalize the commutation result to $m$ Kronecker factor matrices for arbitrary permutations $\sigma$. Early versions of that result can be found in \cite{holmquist_direct_1985}. We use a recent result in \cite{d2017shuffling}, as it is closer in notation to our case, and also delivers some useful corollaries we will make us of.
\begin{theorem}[Commutation of Kronecker Products, Theorem 3.12. in \cite{d2017shuffling}]
\label{theorem:dangeli}
For each $i = 1,\ldots,m$, let $A_i = \left\{a_{i, kl}\right\} \in \mathcal{M}_{r_i \times c_i}\left(\mathbb{R}\right)$ be a matrix with row-index $k=1,\ldots,r_i$ and column-index $l=1,\ldots,c_i$. Further let $\sigma \in Sym\left(m\right)$, then there exist so-called shuffle matrices $L$ and $R$ such that
\begin{equation}
L \left(A_1 \otimes A_2 \otimes \ldots \otimes A_{m} \right) R = A_{\sigma^{-1}\left(1\right)} \otimes A_{\sigma^{-1}\left(2\right)} \otimes \ldots \otimes A_{\sigma^{-1}\left(m\right)},
\end{equation}
where the shuffle matrices are given by
\begin{align}
 L = & \sum_{\substack{i_j = 1,\ldots r_j \\ j = 1,\ldots,k}} E^{i_{\sigma^{-1}\left(1\right)},i_1}_{r_{\sigma^{-1}\left(1\right)}\times r_1} \otimes E^{i_{\sigma^{-1}\left(2\right)},i_2}_{r_{\sigma^{-1}\left(2\right)}\times r_2} \otimes \ldots \otimes E^{i_{\sigma^{-1}\left(m\right)},i_{m}}_{r_{\sigma^{-1}\left(m\right)}\times r_{m}} \\
 R = & \sum_{\substack{j_h = 1,\ldots c_h \\ h = 1,\ldots,k}} E^{j_{1},j_{\sigma^{-1}\left(1\right)}}_{c_{1}\times c_{\sigma^{-1}\left(1\right)}} \otimes E^{j_{2},j_{\sigma^{-1}\left(2\right)}}_{c_{2}\times c_{\sigma^{-1}\left(2\right)}} \otimes \ldots \otimes E^{j_{m},j_{\sigma^{-1}\left(m\right)}}_{c_{m}\times c_{\sigma^{-1}\left(m\right)}}
\end{align}
and $E^{i,j}_{m \times n}$ is a elementary matrices of size $m \times n$ and a 1 at position $\left(i,j\right)$, and zeros otherwise.
\end{theorem}
Similarly useful will be corollary 4.2 in \cite{d2017shuffling}, which we will present here in slightly different form.
\begin{corollary}[Corollary 4.2. in \cite{d2017shuffling}]
\label{cor:dangeli}
Define
\begin{equation}
K^{\sigma}_r := \sum_{\substack{i_j = 1,\ldots r_j \\ j = 1,\ldots,k}} E^{i_{\sigma^{-1}\left(1\right)},i_1}_{r_{\sigma^{-1}\left(1\right)}\times r_1} \otimes E^{i_{\sigma^{-1}\left(2\right)},i_2}_{r_{\sigma^{-1}\left(2\right)}\times r_2} \otimes \ldots \otimes E^{i_{\sigma^{-1}\left(m\right)},i_{m}}_{r_{\sigma^{-1}\left(m\right)}\times r_{m}},
\end{equation}
where $r = \left(r_1,\ldots,r_m\right)$ and $\sigma \in Sym(m)$. Then it follows that
\begin{equation}
K^{\sigma^{-1}}_{r_{\sigma^{-1}}} K^{\sigma}_r = I_{\prod_i r_i},
\end{equation}
with $r_{\sigma^{-1}} = \left(r_{\sigma^{-1}(1)},\ldots,r_{\sigma^{-1}(m)}\right)$.
\end{corollary}
$K^{\sigma}_r$ is identical to $P^{\sigma}_r$ in \cite{d2017shuffling}. We use $K$ instead of $P$ to emphasize the similarity to commutation matrices.
Lastly, note that matrices $L$ and $R$ in Theorem \ref{theorem:dangeli} are then given by $L = P^{\sigma}_r$ and $R = P^{\sigma^{-1}}_{c_{\sigma^{-1}}}$.

We can now proceed to prepare the proof of step 1 using the following lemma.

\begin{lemma}
\label{eq:invariant_perm}
For a permutation $\sigma \in Sym(k)$ and for matrices
\begin{align}
D_k & = D^{\otimes k} \\
X_n & = X^{\otimes n}, \\
\end{align}
where $D \in \mathcal{M}_{1 \times n_y}\left(\R\right)$ and $X \in \mathcal{M}_{n_x \times 1}\left(\R\right)$,
there exist shuffle matrices $K^{\sigma^{-1}}_{c_{\sigma^{-1}}}$ and $K^{\sigma}_r$ such that
\begin{equation}
D_kK^{\sigma^{-1}}_{c_{\sigma^{-1}}} = D_k
\end{equation}
with
\begin{equation}
K^{\sigma^{-1}}_{c_{\sigma^{-1}}}= \sum_{\substack{j_h = 1,\ldots n_y \\ h = 1,\ldots,k}}  E^{j_{1},j_{{\sigma}^{-1}\left(1\right)}}_{n_y\times n_y}  \otimes E^{j_{2},j_{{\sigma}^{-1}\left(2\right)}}_{c}  \otimes \ldots \otimes E^{j_{k},j_{{\sigma}^{-1}\left(k\right)}}_{n_y\times n_y}
\end{equation}
and
\begin{equation}
K^{\sigma}_r X_n  = X_n
\end{equation}
with
\begin{equation}
K^{\sigma}_r = \sum_{\substack{i_j = 1,\ldots r_j \\ j = 1,\ldots,k}}  E^{i_{\sigma^{-1}\left(1\right)},i_1}_{r_{{\sigma^{-1}\left(1\right)}}\times r_1}  \otimes E^{i_{\sigma^{-1}\left(2\right)},i_2}_{r_{{\sigma^{-1}\left(2\right)}}\times r_2}  \otimes \ldots \otimes E^{i_{\sigma^{-1}\left(k\right)},i_{k}}_{r_{{\sigma^{-1}\left(k\right)}}\times r_{k}},
\end{equation}

\end{lemma}
\begin{proof}
Notice that $D_k  = D \otimes \ldots \otimes D$ is invariant under any permutation ${\sigma}$, as all factors are identical. We can then write down a shuffle matrix $K^{\sigma^{-1}}_{c_{\sigma^{-1}}}$ following Theorem \ref{theorem:dangeli} by
\begin{equation}
K^{\sigma^{-1}}_{c_{\sigma^{-1}}}= \sum_{\substack{j_h = 1,\ldots n_y \\ h = 1,\ldots,k}}  E^{j_{1},j_{{\sigma}^{-1}\left(1\right)}}_{n_y\times n_y}  \otimes E^{j_{2},j_{{\sigma}^{-1}\left(2\right)}}_{n_y\times n_y}  \otimes \ldots \otimes E^{j_{k},j_{{\sigma}^{-1}\left(k\right)}}_{n_y\times n_y},
\end{equation}
since $c_j = n_y$ for all $j$. Because $D_k$ is a row vector, the left-handed shuffle matrix is just 1 and then by construction it holds that $D_k K^{\sigma^{-1}}_{c_{\sigma^{-1}}} = D_k$.
\\\\
A similar argument can be made for $X_n$. The only difference here is that we don't want to shuffle all individual $n$ matrices but only a $k$ matrix decomposition out of all $n$ matrices. Given a permutation $\sigma$ the shuffle is then given by
\begin{equation}
 K^{\sigma}_r = \sum_{\substack{i_j = 1,\ldots r_j \\ j = 1,\ldots,k}}  E^{i_{\sigma^{-1}\left(1\right)},i_1}_{r_{{\sigma^{-1}\left(1\right)}}\times r_1}  \otimes E^{i_{\sigma^{-1}\left(2\right)},i_2}_{r_{{\sigma^{-1}\left(2\right)}}\times r_2}  \otimes \ldots \otimes E^{i_{\sigma^{-1}\left(k\right)},i_{k}}_{r_{{\sigma^{-1}\left(k\right)}}\times r_{k}},
\end{equation}
where rows $r_j$ depend on the choice of the $k$ matrix decomposition of $n$.
Since $X_n$ is a column vector the right-handed shuffle matrix is 1 and since all matrices are identical, $X_n$ is invariant to any shuffle matrix, so it follows
that $X_n =  K^{\sigma}_r  X_n$.

\end{proof}

We proceed by defining the notion of a partial Bell base polynomial. This will be helpful, as these represent the smallest meaningful part of a Bell polynomial to which we can apply commutation matrices.

\begin{definition}[Partial Bell Base Polynomial]
We a define a partial Bell `base polynomial' using
\begin{equation}
\tilde{\textbf{B}}^g_{n,k,j}:= g_{x}^{\otimes j_1} \otimes g_{x^2}^{\otimes j_2} \otimes \ldots \otimes g_{x^{n-k+1}}^{\otimes j_{n-k+1}},
\end{equation}
where $j = \left(j_1,\ldots,j_{n-k+1}\right)$.
\end{definition}

Note that $\tilde{\textbf{B}}^g_{n,k,j}$ is a Kronecker product of $k$ matrices because of
\begin{align}
\tilde{\textbf{B}}^g_{n,k,j}:=  & \underbrace{g_{x} \otimes \ldots \otimes g_{x}}_{j_1 \; \text{times}} \otimes  \ldots \otimes  \underbrace{g_{x^{n-k+1}} \otimes \ldots \otimes g_{x^{n-k+1}}}_{ j_{n-k+1} \; \text{times}} \\
:=& g^{\left(1\right)} \otimes \ldots \otimes g^{\left(j_1\right)} \otimes  \ldots \otimes  g^{\left(k-j_{n-k+1}+1\right)} \otimes \ldots \otimes g^{\left(k\right)}
\end{align}
using the fact that $\sum_k j_l = k$.

\begin{example}
With $n = 3, k = 2, j = \left(1,1\right)$, we get
\begin{equation}
\tilde{\textbf{B}}^g_{3,2, \left(1,1\right)} = g_x \otimes g_{x^2}.
\end{equation}
Strictly speaking, we could omit the subindices for $n$ and $k$ in the base polynomial notation. However, it turns out to be useful to be explicitly keeping track of which partial Bell polynomial a specific base polynomial belongs to, without computing that using values in $j$. For example:
\begin{equation}
g_x \otimes \tilde{\textbf{B}}^g_{3,2, \left(1,1\right)} = \tilde{\textbf{B}}^g_{4,3, \left(2,1\right)},
\end{equation}
because $j_1 + j_2 = 3$ and $j_1 + 2j_2 = 4$.
\end{example}

\begin{definition}[Permutations of Partial Bell Base Polynomials]
Given a permutation $\sigma \in Sym(k)$ over these $k$ matrix factors , then $\tilde{\textbf{B}}^{g,\sigma}_{n,k,j}$ describes the Bell base polynomial under a permutation $\sigma$
\begin{equation}
\tilde{\textbf{B}}^{g,\sigma}_{n,k,j}:= g^{\sigma\left(1\right)} \otimes \ldots \otimes g^{\sigma\left(j_1\right)} \otimes  \ldots \otimes  g^{\sigma\left(k-j_{n-k+1}+1\right)} \otimes \ldots \otimes g^{\sigma\left(k\right)}
\end{equation}
Based on the previous discussion, it should be immediately clear that in general $\tilde{\textbf{B}}^{g}_{n,k,j}  \neq \tilde{\textbf{B}}^{g,\sigma}_{n,k,j}$.
\end{definition}

We can justify calling $\tilde{\textbf{B}}^g_{n,k,j}$ `base polynomials' because partial Bell polynomials can be represented as linear combinations
\begin{equation}
\textbf{B}_{n,k}^g = \sum_j \alpha_j \tilde{\textbf{B}}^g_{n,k,j},
\end{equation}
where $j = \left(j_1,\ldots,j_{n-k+1}\right)$ and $\alpha_j = \frac{n!}{j_1! j_2! \ldots j_{n-k+1}! \; 2^{j_2} 3^{j_3} (n-k+1)^{j_{n-k+1}}}$.
The multi-index $j$ iterates over all vectors which satisfy $\sum_l j_l = k$ and $\sum_l l j_l = n$.

Based in this definition and using the previous lemma we are now able to show that under a simple transformation Bell base polynomial are invariant to permutations of their ordering.

\begin{theorem}[Proof of Step 1, Order Invariance of Partial Bell Base Polynomials]
\label{theorem:kontny}
For Bell base polynomials
\begin{equation}
\tilde{\textbf{B}}^g_{n,k,j}:= g_{x}^{\otimes j_1} \otimes g_{x^2}^{\otimes j_2} \otimes \ldots \otimes g_{x^{n-k+1}}^{\otimes j_{n-k+1}},
\end{equation}
where $j = \left(j_1,\ldots,j_{n-k+1}\right)$ with $\sum_l j_l = k$ and $\sum_l l j_l = n$, we can shuffle the component matrices arbitrarily such that
\begin{equation}
D_k \tilde{\textbf{B}}^g_{n,k,j} X_n = D_k \tilde{\textbf{B}}^{g,\sigma}_{n,k,j} X_n,
\end{equation}
where $\tilde{\textbf{B}}^{g,\sigma}_{n,k,j}$ is a reshuffling of individual matrices in $\tilde{\textbf{B}}^g_{n,k,j}$ under some permutation $\sigma$ over $\left\{1,\ldots,k\right\}$.
\end{theorem}
\begin{proof}
Take the Bell base polynomial of the form
\begin{equation}
\tilde{\textbf{B}}^g_{n,k,j}:= g_{x}^{\otimes j_1} \otimes g_{x^2}^{\otimes j_2} \otimes \ldots \otimes g_{x^{n-k+1}}^{\otimes j_{n-k+1}},
\end{equation}
where $j = \left(j_1,\ldots,j_{n-k+1}\right)$ and $\sum_l j_l = k$ and $\sum_l l j_l = n$. $\tilde{\textbf{B}}^g_{n,k,j}$ now consists of exactly $k$ individual matrices of the form $g_{x^l}$  for $l = 1,\ldots,n-k+1$. Each product component has $n_y$ rows and $n_x^l$ columns. Given some permutation $\sigma$ over $\left\{1,\ldots,k\right\}$, Theorem \ref{theorem:dangeli} delivers shuffle matrices $K^{\sigma}_r $ and $K^{\sigma^{-1}}_{c}$:

\begin{equation}
K^{\sigma}_r   = \sum_{\substack{i_j = 1,\ldots n_y \\ j = 1,\ldots,k}}  E^{i_{\sigma^{-1}\left(1\right)},i_1}_{n_y\times n_y}  \otimes E^{i_{\sigma^{-1}\left(2\right)},i_2}_{n_y\times n_y}  \otimes \ldots \otimes E^{i_{\sigma^{-1}\left(k\right)},i_{k}}_{n_y\times n_y}
\end{equation}
with
\begin{equation}
\left(r_1,r_2,\ldots,r_{k}\right) = \left(\underbrace{n_y,\ldots,n_y}_{k \; \text{times}}\right).
\end{equation}

\begin{equation}
K^{\sigma^{-1}}_{c_{\sigma^{-1}}}  = \sum_{\substack{j_h = 1,\ldots c_h \\ h = 1,\ldots,k}}  E^{j_{1},j_{\sigma^{-1}\left(1\right)}}_{c_{1}\times c_{\sigma^{-1}\left(1\right)}} \otimes E^{j_{2},j_{\sigma^{-1}\left(2\right)}}_{c_{2}\times c_{\sigma^{-1}\left(2\right)}} \otimes \ldots \otimes E^{j_{k},j_{\sigma^{-1}\left(k\right)}}_{c_{k}\times c_{\sigma^{-1}\left(k\right)}},
\end{equation}
with
\begin{equation}
\left(c_1,c_2,\ldots,c_{k}\right) = \left(\underbrace{n_x,\ldots,n_x}_{j_1 \; \text{times}}, \ldots, \underbrace{n_x^l,\ldots,n_x^l}_{j_l \; \text{times}}, \ldots, \underbrace{n_x^{n-k+1},\ldots,n_x^{n-k+1}}_{j_{n-k+1} \; \text{times}}\right).
\end{equation}

The resulting Kronecker product of matrices using a permutation $\sigma$ is then given by
\begin{equation}
\tilde{\textbf{B}}^{g,\sigma}_{n,k,j} =  K^{\sigma}_r   \tilde{\textbf{B}}^{g}_{n,k,j} K^{\sigma^{-1}}_{c}.
\end{equation}

Following Lemma \ref{eq:invariant_perm} we can choose columns of the shuffle matrix of $D_k$, i.e. $K^{\sigma^{-1}}_{c_{\sigma^{-1}}}$ such that $c = r$. This is possible, because in this case the number of rows of all base polynomial factors, $g_{x^l}$ is $n_y$ as is the number of columns of $D_k$. Using the same permutation $\sigma$ we therefore get using Corrolary \ref{cor:dangeli}
\begin{equation}
K^{\sigma^{-1}}_{r_{\sigma^{-1}}} K^{\sigma}_r = I_{n_y^k}.
\end{equation}

Now for $X_n = X \otimes \cdots \otimes X$ choose a decomposition of $k$ matrices such that their rows line up with columns $\left(c_1,c_2,\ldots,c_k\right)$ of the $k$ factors of $\tilde{\textbf{B}}^{g}_{n,k,j}$.
This is possible since each individual $X$ is a column vector of size $n_x \times 1$, such that we can write for a given choice of indices $j$:
\begin{equation}
X_n = \underbrace{\left(X \otimes \ldots \otimes X\right)}_{j_1 \; \text{times}} \otimes \ldots \otimes \underbrace{\left(X^{\otimes l} \otimes \ldots \otimes X^{\otimes l}\right)}_{j_l \; \text{times}} \otimes \ldots \otimes \underbrace{\left(X^{\otimes n-k+1} \otimes \ldots \otimes X^{\otimes n-k+1}\right)}_{j_{n-k+1} \; \text{times}}.
\end{equation}
Because by definition $\sum_l l j_l = n$ and $\sum_l j_l = k$ and because each individual $X^{\otimes l}$ has $n_x^l$ rows, its rows are given by
\begin{equation}
\label{eq:defr}
\left({r}_1,{r}_2,\ldots,{r}_{k}\right) = \left(\underbrace{n_x,\ldots,n_x}_{j_1 \; \text{times}}, \ldots, \underbrace{n_x^l,\ldots,n_x^l}_{j_l \; \text{times}}, \ldots, \underbrace{n_x^{n-k+1},\ldots,n_x^{n-k+1}}_{j_{n-k+1} \; \text{times}}\right).
\end{equation}

Now choose a shuffle matrix for the same permutation $\sigma$ as before given by Lemma \ref{eq:invariant_perm} with $K^{\sigma}_{c}$ where we choose $c = r$, where $r$ is given by definition \eqref{eq:defr}.
By Corrolary \ref{cor:dangeli} it follows that $K^{\sigma^{-1}}_{c_{\sigma^{-1}}} K^{\sigma}_c = I_{n_x^n}$.
Hence we get
\begin{align}
   D_k \tilde{\textbf{B}}^{g}_{n,k,j} X_n  &= (D_k K^{\sigma^{-1}}_{r_{\sigma^{-1}}}) ( K^{\sigma}_r \tilde{\textbf{B}}^{g}_{n,k,j} K^{\sigma^{-1}}_{c_{\sigma^{-1}}} ) (K^{\sigma}_c X_n).  \\
  																																						& = D_k \left( K^{\sigma^{-1}}_{r_{\sigma^{-1}}}    K^{\sigma}_r\right) \tilde{\textbf{B}}^{g,\sigma}_{n,k,j}\left( K^{\sigma^{-1}}_{c_{\sigma^{-1}}}   K^{\sigma}_c \right) X_n \\
 																																							& =  D_k \tilde{\textbf{B}}^{g,\sigma}_{n,k,j} X_n,
\end{align}
which concludes the proof.
\end{proof}
Having shown the invariance of Bell base polynomial under pre- and post-multiplication of $D_k$ and $X_n$, we can proceed to prove that partial Bell recurrence relations can similarly be written in that way.

\begin{proposition}[Proof of Step 2, Invariance of Bell Recurrence Relations]
The following partial Bell polynomial recurrences hold
\label{prop:bellequality_multi}
 \begin{enumerate}
\item[i)] $D_k \textbf{B}_{n,k}^g X_n = D_k \left[ \sum_{i = 1}^{n-k+1} {n-1 \choose i-1} \left(\textbf{B}_{n-i,k-1}^g \otimes g_{x^i}\right) \right]  X_n$ \\
\item[ii)]$ D_k\frac{d}{dx} \textbf{B}_{n,k}^g X_{n+1} = D_k \left[ \sum_{i = 1}^{n-k+1} {n \choose i } \left( \textbf{B}_{n-i,k-1}^g \otimes g_{x^{i+1}} \right) \right] X_{n+1}$ \\
 \item[iii)]  $D_k \textbf{B}^g_{n+1,k} X_{n+1} = D_k \left[ \left( \textbf{B}^g_{n,k-1} \otimes g_{x} \right)  +\frac{\partial \textbf{B}^g_{n,k}}{\partial x'}\right] X_{n+1}$\\
 \end{enumerate}
 \end{proposition}
\begin{proof}
Here we will prove recurrence relation i) explicitly. Recurrence ii) then follows using the identical arguments and additionally using the extended Kronecker product rule in Appendix \ref{eq:kronprodrule}. The third recurrence then follows from Lemma \ref{lemma:bellrecur} and recurrences i) and ii), as the proof of Lemma \ref{lemma:bellrecur} does not require that the Kronecker product commutes. The general idea is to show that using the previously established results on commutation of Kronecker products in Theorem \ref{theorem:kontny}, these recurrences hold in the multivariate case exactly because they also hold in the univariate case. To start, recall that the partial Bell polynomial can be written as
\begin{equation}
\textbf{B}_{n,k}^g = \sum_j \alpha_j  \tilde{\textbf{B}}^{g}_{n,k,j},
\end{equation}
where $j = \left(j_1,\ldots,j_{n-k+1}\right)$ and $\alpha_j = \frac{n!}{j_1! j_2! \ldots j_{n-k+1}! \; 2^{j_2} 3^{j_3} (n-k+1)^{j_{n-k+1}}}$. Multi-index $j$ iterates over all vectors which satisfy that $\sum_l j_l = k$ and $\sum_l l j_l = n$.
  \vfill
It follows that we can write for all $i \in \left\{1,2,\ldots,n-i-k+2\right\}$
\begin{align}
      \textbf{B}_{n-i,k-1}^g \otimes g_{x^i} & =    \left(\sum_{j^{(i)}} \alpha_{j^{(i)}} \tilde{\textbf{B}}^{g}_{n-i,k-1,j^{(i)}}\right) \otimes g_{x^i}\\
 & = \sum_{j^{(i)}} \alpha_{j^{(i)}} \left( \tilde{\textbf{B}}^{g}_{n-i,k-1,j^{(i)}} \otimes g_{x^i}\right)\\
 & = \sum_{j^{(i)}} \alpha_{j^{(i)}} \left(g_{x}^{\otimes j^{(i)}_1} \otimes g_{x^2}^{\otimes j^{(i)}_2} \otimes \ldots \otimes g_{x^{n-i-k+2}}^{\otimes j^{(i)}_{n-i-k+2}} \otimes g_{x^i} \right),
\end{align}
with $j^{(i)} = \left(j^{(i)}_1,\ldots,j^{(i)}_{n-i-k+2}\right)$ and $\sum_{l} j^{(i)} = k-1$ and $\sum_l l j^{(i)} = n-i$ . For any $i$, Theorem \ref{theorem:dangeli} provides matrices $K_{r_i}^{\sigma}$ and $K_{c_{i,\sigma^{-1}}}^{\sigma^{-1}}$ such that
\begin{align}
& g_{x}^{\otimes j^{(i)}_{1}} \otimes\cdots \otimes g_{x^{n-i-k+2}}(x)^{\otimes j^{(i)}_{n-i-k+2}} \otimes g_{x^i} \\
= &K_{r_i}^{\sigma} \left( g_{x}^{\otimes j^{(i)}_{1}}\otimes \cdots \otimes  g_{x^{i}}^{\otimes (j^{(i)}_{i}+1)}\otimes\cdots \otimes g_{x^{n-i-k+2}}(x)^{\otimes j^{(i)}_{n-i-k+2}} \right) K_{c_{i,\sigma^{-1}}}^{\sigma^{-1}},
\end{align}
where $r_i$ and $c_i$ are the respective vectors of rows and columns.
Because $D_k K_{c_{i,\sigma^{-1}}}^{\sigma^{-1}} = D_k$ and $K_{r_i}^{\sigma} X_n = X_n$ for all $i$, we can apply Theorem \ref{theorem:kontny}, which then delivers for the RHS of recurrence i) the following:

\begin{align}
  &D_k\left[ \sum_{i = 1}^{n-k+1} {n-1 \choose i-1} \left(B_{n-i,k-1}^g \otimes g_{x^i}\right) \right]  X_n \\
=&D_k\left[ \sum_{i = 1}^{n-k+1} {n-1 \choose i-1} \sum_{j^{(i)}} \alpha_{j^{(i)}} \left( g_{x}^{\otimes j^{(i)}_{1}}\otimes \cdots  g_{x^{i}}^{\otimes (j^{(i)}_{i}+1)}\otimes\cdots \otimes g_{x^{n-i-k+2}}(x)^{\otimes j^{(i)}_{n-i-k+2}} \right)\right]  X_n. \\
\end{align}

Now, notice that for each $i$, we can write
\begin{align}
 \tilde{\textbf{B}}^{g}_{n,k,\tilde{j}^{(i)}} & := g_{x}^{\otimes j_{1}}\otimes \cdots  g_{x^{i}}^{\otimes (j_{i}+1)}\otimes\cdots \otimes g_{x^{n-i-k+2}}(x)^{\otimes j_{n-i-k+2}}\\
 & = g_{x}^{\otimes \tilde{j}^{\left(i\right)}_{1}}\otimes \cdots  g_{x^{i}}^{\otimes \tilde{j}^{\left(i\right)}_{i}}\otimes\cdots \otimes g_{x^{n-i-k+2}}(x)^{\otimes \tilde{j}^{\left(i\right)}_{n-i-k+2}}
\end{align}
with $\tilde{j}^{\left(i\right)}$ given by
\begin{equation}
\tilde{j}^{\left(i\right)} = \left(\tilde{j}^{\left(i\right)}_1, \ldots ,\tilde{j}^{\left(i\right)}_i,\ldots \tilde{j}^{\left(i\right)}_{n-k+1 + (1-i)},\textbf{0}\right) = \left(j^{(i)}_1, \ldots ,j^{(i)}_i +1,\ldots j^{(i)}_{n-k+1 + (1-i)},\textbf{0}\right),
\end{equation}
where bold $\textbf{0}$ is a row vector of zeros such that $\tilde{j}^{\left(i\right)}$ is of size $1 \times (n-k+1)$.
It easy to check that for multi-indices $\tilde{j}^{\left(i\right)}$ it now must hold that
\begin{equation}
\sum_{l} \tilde{j}^{\left(i\right)}_l = k \quad \text{and} \quad \sum_{l} l \tilde{j}^{\left(i\right)}_l = n,
\end{equation}
justifying writing the base polynomial as $ \tilde{\textbf{B}}^{g}_{n,k,\tilde{j}^{(i)}}$.
\\\\

Hence, if follows
\begin{align}
  &D_k\left[ \sum_{i = 1}^{n-k+1} {n-1 \choose i-1} \left(g_{x^i} \otimes \textbf{B}_{n-i,k-1}^g\right) \right] X_n \\
=&D_k\left[ \sum_{i = 1}^{n-k+1} {n-1 \choose i-1} \sum_{j^{(i)}} \alpha_{j^{(i)}}  \tilde{\textbf{B}}^{g}_{n,k,\tilde{j}^{(i)}}\right]  X_n\\
=&D_k\left[ \sum_{i = 1}^{n-k+1} \sum_{j^{(i)}} {n-1 \choose i-1}  \alpha_{j^{(i)}} \tilde{\textbf{B}}^{g}_{n,k,\tilde{j}^{(i)}}\right] X_n. \label{eq:thmrhs}
\end{align}

Next, the left-hand-side (LHS) of i) can be written as
\begin{equation}
D_k \textbf{B}_{n,k}^g X_n =  D_k  \sum_{j} \alpha_{j} \tilde{\textbf{B}}^{g}_{n,k,{j}^{}} X_n, \label{eq:thmlhs}
\end{equation}
where $j$ is such that
\begin{equation}
\sum_{l} j_l = k \quad \text{and} \quad \sum_{l} l j_l = n.
\end{equation}

Therefore, by combining equations \eqref{eq:thmlhs} and \eqref{eq:thmrhs} the following must hold for i) to be true
\begin{equation}
D_k  \left[\sum_{j} \alpha_{j} \tilde{\textbf{B}}^{g}_{n,k,{j}^{}}\right] X_n = D_k\left[ \sum_{i = 1}^{n-k+1} \sum_{j^{(i)}} {n-1 \choose i-1}  \alpha_{j^{(i)}} \tilde{\textbf{B}}^{g}_{n,k,\tilde{j}^{(i)}}\right] X_n.
\end{equation}
Since $\tilde{\textbf{B}}^{g}_{n,k,{j}^{}}$ and $\tilde{\textbf{B}}^{g}_{n,k,\tilde{j}^{(i)}}$ have the same matrix size, it would suffice to show that the rearrangement of the double-sum on the RHS leads to the sum on the LHS.

However, we know that exactly this must be the case, since in the univariate case it holds that
\begin{equation}
\sum_{j} \alpha_{j} \tilde{{B}}^{g}_{n,k,{j}^{}} =  \sum_{i = 1}^{n-k+1} \sum_{j^{(i)}} {n-1 \choose i-1}  \alpha_{j^{(i)}}  \tilde{{B}}^{g}_{n,k,\tilde{j}^{(i)}},
\end{equation}
where the non-bold Bell base polynomials $\tilde{{B}}^{g}_{n,k,{j}^{}}$ and $\tilde{{B}}^{g}_{n,k,\tilde{j}^{(i)}}$ represent their univariate counterparts.
The multivariate recurrence relation i) follows.
\end{proof}
%
%
\section{Generalized Multivariate Faà Di Bruno's Formula}
We now prove our multivariate and matrix-valued version of Faà Di Bruno's formula by using the invariance of recurrence relations shown in Proposition \ref{prop:bellequality_multi}.
To do that, we must first identify matrices $D_k$ and $X_n$ for use in Faà Di Bruno.
\begin{lemma}
\label{prepFaa}
For composite functions $\left(f \circ g\right): \R^{n_x} \rightarrow \R^{n_f}$, by setting $D_k = f_{g^k}$ and $X_{n+1} = \left(dx\right)^{\otimes n+1}$ it follows
\begin{equation}
 f_{g^k} \textbf{B}_{n+1,k}^g \left(dx\right)^{\otimes n+1} =  f_{g^k}  \left[ \left(  g_x \otimes \textbf{B}^g_{n,k-1} \right)  +\frac{\partial \textbf{B}^g_{n,k}}{\partial x'}\right] \left(dx\right)^{\otimes n+1}
\end{equation}
\end{lemma}
\begin{proof}
The statement follows from Proposition \ref{prop:bellequality_multi} by setting $D_k =  f_{g^k}$ and $X_{n+1}=  \left(dx\right)^{\otimes n+1}$.
Choosing $X = dx$, which is a $n_x \times 1$ column-vector, satisfies the condition in Lemma \ref{eq:invariant_perm}. For $D_k = f_{g^k}$ we can write the derivative as
\begin{equation}
f_{g^k} = f \otimes \underbrace{{\partial \over \partial g'} \otimes \cdots \otimes {\partial \over \partial g'}}_{k \quad \text{times}}.
\end{equation}
Setting $D = {\partial \over \partial g'}$, we get $f \otimes D^{\otimes k}$. We know that any permutation of base polynomials in $ \textbf{B}_{n+1,k}^g$ will be invariant with respect to $D^{\otimes k}$. However the same holds with respect to $f \otimes D^{\otimes k}$. This follows from Lemma \ref{eq:invariant_perm} with
\begin{align}
f \otimes D^{\otimes k} &= f \otimes D^{\otimes k} K^{\sigma^{-1}}_{c} \\
&= \left(f\cdot 1\right) \otimes \left(D^{\otimes k} K^{\sigma^{-1}}_{c}\right)\\
 & = \left(f \otimes D^{\otimes k}\right)\left(1 \otimes K^{\sigma^{-1}}_{c}\right)\\
& = \left(f \otimes D^{\otimes k}\right) K^{\sigma^{-1}}_{c}.
\end{align}

\end{proof}
Taking everything together that we have shown so far, multivariate Faà Di Bruno now follows using a similar proof of induction as in the univariate case.
\begin{theorem}[Multivariate Faà Di Bruno]
For some $n,k \in \mathbb{N}$ with $n \geq k$ and function $g:\R^{n_x} \rightarrow \R^{n_y}$ which is at least $n-k+1$ differentiable on some domain $D_g \subset \R^{n_x}$ and a function $f:\R^{n_y} \rightarrow \R^{n_f} $ which is at least $n$ times differentiable on some domain $D_f \subset \R^{n_y}$, the $n$-th order total derivative of $f \circ g$ is given by Faà Di Bruno's formula:
\begin{align}\label{FaaDiBruno-Bell-multivariate}
d^{n} f(g(x))
= \sum_{k=1}^{n} f_{g^{k}}(g(x)) \textbf{B}^g_{n,k}  (dx)^{\otimes n},
\end{align}
 where $f_{g^{k}}(g(x))$ denotes the $n_f \times n_{y}^{k}$ matrix of the $k$th-order partial derivatives of $f(\cdot)$ and similarly $g_{x^{l}}(x)$ denotes the $n_{y} \times n_{x}^{l}$ matrix of the $l$th-order partial derivatives of $g(\cdot)$.
\end{theorem}
\begin{proof} We prove the multivariate case again by induction. Using the apparatus set up in this paper, the proof follows closely the univariate one. We leave the induction beginning to the reader and directly start by assuming that \eqref{FaaDiBruno-Bell-multivariate} holds for some arbitrary $n \in \mathbb{N}$. Then it follows using the differential version of the product rule in Appendix A \eqref{eq:diffproduct} and the differential version of the chain rule in Appendix A \eqref{eq:diffchain}:
\begin{align}
d^{n+1} f(g(x)) &= d ( d^n f(g(x)) ) \\
& = d(\sum_{k=1}^{n} f_{g^{k}}(g(x)) \textbf{B}^g_{n,k}  (dx)^{\otimes n}) \\
& = \sum_{k=1}^{n} d( f_{g^{k}}(g(x)) \textbf{B}^g_{n,k}  (dx)^{\otimes n}) \\
& = \sum_{k=1}^{n} f_{g^{k+1}}\left(I_{n_y^k} \otimes g_x\right) \left(\textbf{B}^g_{n,k}  (dx)^{\otimes n} \otimes dx\right) + f_{g^k} \frac{\partial \textbf{B}^g_{n,k}}{\partial x'} \left(dx\right)^{\otimes n+1}\\
& = \sum_{k=1}^{n} f_{g^{k+1}} \left(\textbf{B}^g_{n,k} \otimes g_x\right)\left(dx\right)^{\otimes n+1} + f_{g^k} \frac{\partial \textbf{B}^g_{n,k}}{\partial x'} \left(dx\right)^{\otimes n+1}
\end{align}
By factoring for $f_{g^{k}}\left(g(x)\right)$, re-parameterizing the sums and noting that $\textbf{B}^g_{n,0}\left(\cdot\right) = \textbf{B}^g_{n,n+1}\left(\cdot\right)$ = 0, we get
\begin{align}
& = \sum_{k=1}^{n+1} f_{g^{k}} \left[ \left(\textbf{B}^g_{n,k-1} \otimes g_x\right)+\frac{\partial \textbf{B}^g_{n,k}}{\partial x'}\right] \left(dx\right)^{\otimes n+1}.
\end{align}
Lemma \ref{prepFaa} then provides, that for each $k$
\begin{equation}
  \sum_{k=1}^{n+1} f_{g^{k}}  \textbf{B}^g_{n+1,k}\left(dx\right)^{\otimes n+1} = \sum_{k=1}^{n+1} f_{g^{k}} \left[ \left(\textbf{B}^g_{n,k-1} \otimes g_x\right)+\frac{\partial \textbf{B}^g_{n,k}}{\partial x'}\right] \left(dx\right)^{\otimes n+1}
\end{equation}
holds, which completes the proof.
\end{proof}
\section{Applications}

\subsection{Computing Moments of Multivariate Distributions}
For a $k \times 1$ random variable $X$ we can compute its moments using the moment generating function (MGF)
\begin{equation}
M_X\left(t\right)  = E_X \left[\exp{t' X}\right],
\end{equation}
where $t$ is a $k \times 1$ vector.
The resulting series expansion can then be written as
\begin{equation}
M_X\left(t\right) = \sum_{k=0}^{\infty} {1 \over k!} m_k t^{\otimes k},
\end{equation}
with $m_n = E\left[\left(X'\right)^{\otimes n}\right]$.
The moments of order $n$ are then given by
\begin{equation}
m_n = \left.{\partial^n M_X \over \partial t^{\otimes n}}\right|_{t=\textbf{0}}.
\end{equation}


\paragraph{Example: (Multivariate Normal Distribution)}
For a $k \times 1$ random variable $X$ which is normally distributed with mean vector $\mu$ and covariance matrix $\Sigma$, the multivariate MGF can be computed to be
\begin{equation}
M_X\left(t\right) = \exp{\left(t'\mu+{1\over 2} t' \Sigma t\right)}.
\end{equation}
To compute its $n$-th derivative we again make use of our novel version of Faà Di Bruno's formula given by
\begin{align}
d^{n} (f \circ g)
= \sum_{k=1}^{n} f_{g^{k}}(g(x)) \textbf{B}^g_{n,k}  (dx)^{\otimes n}.
\end{align}
However, in this case we will make use not of the differential form but of the resulting matrix of partial derivatives given by
\begin{equation}
{\partial^n (f \circ g) \over \partial (x')^{\otimes n}} = \sum_{k=1}^{n} f_{g^{k}}(g(x)) \textbf{B}^g_{n,k}.
\end{equation}
Note that this representation might not be a unique representation for the matrix of derivatives.\footnote{For more details on this issue, see Appendix B.} We will deal with this issue towards the end of this application.

Now, setting $f(y) = \exp{y}$ with $f:\R \rightarrow \R$ and $g(t) =t'\mu+{1\over 2} t' \Sigma t$ with $g:\R^k \rightarrow \R$, we get $f_{g^j} = \exp{(g)}$ for all $j$ and
\begin{align}
 g_t  &= \mu' + {1 \over 2} {d \left(t' \Sigma t\right) \over dt} \\
g_{t^j}  &= {1 \over 2} {d^j \left(t' \Sigma t\right) \over dt^j}  \quad \text{for all} \quad j \geq 2.
\end{align}
Evaluated at $t=0$, we get $f_{g^j} = 1$ and therefore we can write moments using
\begin{equation}
m_n =\sum_{j=1}^n \left.\textbf{B}_{n,j}^{g\left(t\right)}\right|_{t=0}.
\end{equation}
Since only derivatives of $g$ up to the second order remain non-zero when evaluated at $t=0$ , i.e.
\begin{align}
 \left. g_t \right|_{t=0}  &= \mu' \\
 \left. g_{t^2} \right|_{t=0} &= vec\left(\Sigma\right)' := \left(\sigma_{11}, \sigma_{12}, \sigma_{12}, \sigma_{22}\right) \\
  \left. g_{t^j} \right|_{t=0} &= 0  \quad \text{for all} \quad j \geq 3,
\end{align}
the sum reduces to
\begin{equation}
m_n =\sum_{j=0}^{\left\lfloor{{n \over 2}}\right\rfloor} \textbf{B}_{n,n-j}\left(\mu',vec\left(\Sigma\right)',0,\cdots,0\right).
\end{equation}
Therefore we arrive at
\begin{equation}
m_n =\sum_{j=0}^{\left\lfloor{{n \over 2}}\right\rfloor} {n! \over (n-2j)! (j)! 2^j} \left((\mu')^{\otimes (n-2j)} \otimes \left(vec\left(\Sigma\right)'\right)^{\otimes j}\right).
\end{equation}
Notice that for centered moments with $\mu = 0$, we have to evaluate the expression only for $j = {n \over 2}$. Hence get for even $n$
\begin{equation} \label{eq:everskontny_moments}
m_n = {n! \over {n \over 2}! 2^{n \over 2}} vec(\Sigma)'^{\otimes {n \over 2}}.
\end{equation}
These expressions are almost identical to \cite{holmquist1988moments} except for his use of the symmetrizer matrix $S_{k,n}$ and the use of columns vectors. When defining the representation based on row vectors, the moment vector in \cite{holmquist1988moments} is given by
\begin{equation}
m_n =\sum_{j=0}^{\left\lfloor{{n \over 2}}\right\rfloor} {n! \over (n-2j)! (j)! 2^j} \left((\mu')^{\otimes (n-2j)} \otimes \left(vec\left(\Sigma\right)'\right)^{\otimes j}\right) S_{k,n}.
\end{equation}

In essence the symmetrizer matrix ensures that each component of the moment vector represents the exact value of one of the many possible scalar moments.
To illustrate this point and therefore the relevance of the symmetrizer matrix, we compute $m_4$ with $\mu = 0$ and $k=2$ using equation \eqref{eq:everskontny_moments}, which yields
\begin{equation}
m_4 = 3 vec(\Sigma)'^{\otimes {2}}.
\end{equation}
In the series expansion of the moment generating function we get
\begin{equation}
M_X\left(t\right) = \sum_{k=1}^{\infty} {1 \over k!} m_k t^{\otimes k}.
\end{equation}
Therefore, we show the moment vector $m_4$ together with the corresponding vector $t^{\otimes 4}$:
\begin{equation}
m_4=\left(\begin{array}{c} 3\,{\sigma_{11}}^2\\ 3\,\sigma_{11}\,\sigma_{12}\\ 3\,\sigma_{11}\,\sigma_{12}\\ 3\,\sigma_{11}\,\sigma_{22}\\ 3\,\sigma_{11}\,\sigma_{12}\\ 3\,{\sigma_{12}}^2\\ 3\,{\sigma_{12}}^2\\ 3\,\sigma_{12}\,\sigma_{22}\\ 3\,\sigma_{11}\,\sigma_{12}\\ 3\,{\sigma_{12}}^2\\ 3\,{\sigma_{12}}^2\\ 3\,\sigma_{12}\,\sigma_{22}\\ 3\,\sigma_{11}\,\sigma_{22}\\ 3\,\sigma_{12}\,\sigma_{22}\\ 3\,\sigma_{12}\,\sigma_{22}\\ 3\,{\sigma_{22}}^2 \end{array}\right)', \quad  t^{\otimes 4} = \left(\begin{array}{c} {t_{1}}^4\\ {t_{1}}^3\,t_{2}\\ {t_{1}}^3\,t_{2}\\ {t_{1}}^2\,{t_{2}}^2\\ {t_{1}}^3\,t_{2}\\ {t_{1}}^2\,{t_{2}}^2\\ {t_{1}}^2\,{t_{2}}^2\\ t_{1}\,{t_{2}}^3\\ {t_{1}}^3\,t_{2}\\ {t_{1}}^2\,{t_{2}}^2\\ {t_{1}}^2\,{t_{2}}^2\\ t_{1}\,{t_{2}}^3\\ {t_{1}}^2\,{t_{2}}^2\\ t_{1}\,{t_{2}}^3\\ t_{1}\,{t_{2}}^3\\ {t_{2}}^4 \end{array}\right).
\end{equation}
The elements of $m_4$ corresponding to a specific element of $t^{\otimes 4}$ should deliver the corresponding individual scalar moment. So for example, the entry of $m_4$ corresponding to $t_1^2 t_2^2$ should be equal to the scalar moment $E\left[x_1^2 x_2^2\right]$. However, as the entries of $t_1^2 t_2^2$ appear multiple times, one can notice that the corresponding entries in $m_4$ aren't identical. To arrive at the correct moment we average over all values appearing with $t_1^2 t_2^2$. By averaging over these entries, we get
\begin{equation}
m_4 =  \left(\begin{array}{c} 3\,{\sigma_{11}}^2\\ 3\,\sigma_{11}\,\sigma_{12}\\ 3\,\sigma_{11}\,\sigma_{12}\\ 2\,{\sigma_{12}}^2+\sigma_{11}\,\sigma_{22}\\ 3\,\sigma_{11}\,\sigma_{12}\\ 2\,{\sigma_{12}}^2+\sigma_{11}\,\sigma_{22}\\ 2\,{\sigma_{12}}^2+\sigma_{11}\,\sigma_{22}\\ 3\,\sigma_{12}\,\sigma_{22}\\ 3\,\sigma_{11}\,\sigma_{12}\\ 2\,{\sigma_{12}}^2+\sigma_{11}\,\sigma_{22}\\ 2\,{\sigma_{12}}^2+\sigma_{11}\,\sigma_{22}\\ 3\,\sigma_{12}\,\sigma_{22}\\ 2\,{\sigma_{12}}^2+\sigma_{11}\,\sigma_{22}\\ 3\,\sigma_{12}\,\sigma_{22}\\ 3\,\sigma_{12}\,\sigma_{22}\\ 3\,{\sigma_{22}}^2 \end{array}\right)',
\end{equation}
which implies $E\left[x_1^2 x_2^2\right] = 2 \sigma_{12}^2 + \sigma_{11} \sigma_{22}$.
This moment vector $m_4$ is now identical to the one computed using the formula in \cite{holmquist1988moments}.
The symmetrizer matrix in \cite{holmquist1988moments} ensures that all entries that correspond to the same individual scalar moment deliver the exact same value. The fundamental issue here, is that the matrix of derivative is not uniquely defined by
\begin{equation}
{\partial^n (f \circ g) \over \partial (x')^{\otimes n}}  = \sum_{k=1}^{n} f_{g^{k}}(g(x)) \textbf{B}^g_{n,k}.
\end{equation}
Instead by using
\begin{equation}
{\partial^n (f \circ g) \over \partial (x')^{\otimes n}} = \sum_{k=1}^{n} f_{g^{k}}(g(x)) \textbf{B}^g_{n,k} S_{n_x,n},
\end{equation}
where $S_{n,k}$ is the symmetrizer matrix for n vectors of size $n_x \times 1$, we would arrive at a representation identical to \cite{holmquist1988moments}.\footnote{See Appendix B for more details.} In practice, the computation of the full symmetrizer matrix might not be necessary. In our application, it would suffice to identify the elements of $m_n$ which are not identical with respect to the corresponding elements of $t^{\otimes n}$ and then averaging over all realizations.
\section{Concluding Remarks}
In this paper, we introduce a novel generalization of the partial Bell polynomial. This generalization is similar in structure to the univariate version, except for the replacement of the scalar multiplication operation with the Kronecker product. This extension present us with several challenges, because the Kronecker product doesn't generally commute. Nevertheless, we are able to show that under a simple transformation well-known recurrences relations for the partial Bell polynomial continue to hold. It is exactly because of that insight that we can prove a generalization of Faà Di Bruno's formula which is again similar in structure to the univariate version. In our application we show that using such a representation of higher-order composite derivatives, the computation of higher-order moments of the multivariate normal distribution becomes a simple task. This representation has further potential applications in multivariate statistics and in the theory and computation of dynamic stochastic general equilibrium models in economics, which we attempt to explore in the future.

\newpage

\bibliographystyle{Econometrica}
\bibliography{BellCalculus}

\newpage
\section*{Appendix}

\subsection*{A. Derivatives and Differentials}
In this section, we present and derive representations and derivative rules vector- and matrix-valued derivatives and differentials. In that we mostly follow \cite{magnus_matrix_1999} and \cite{rogers1980matrix}.

We start by defining the vector- and matrix-valued representations of derivative and differentials. For this purpose, let matrices $F,G,H$ be functions of a vector $x$ of size $n_x \times 1$. The matrix $F$ is of size $s \times t$, $G$ of size $y \times 1$, and $H$ of size $p \times q$. We define a differential operator ${\partial \over \partial x'}$ w.r.t the transpose of $x$ by
\begin{equation}
{\partial \over \partial x'} = \left({\partial \over \partial x_1},\ldots,{\partial \over \partial x_{n_x}} \right).
\end{equation}
First-order derivatives of matrices and vectors are then defined by
\begin{equation}
F_x = {\partial F \over \partial x'} := F \otimes {\partial  \over \partial x'}.
\end{equation}
Analogously, higher-order derivatives are defined by
\begin{equation}
 F_{x^k} = {\partial F \over \partial (x')^{\otimes k}} := F \otimes \left({\partial  \over \partial x'}\right)^{\otimes k}.
\end{equation}
This notation implies some structure on the ordering of partial derivative within ${\partial F \over \partial x'}$.\footnote{See \cite{tracy1993higher}, \cite{sultan1996moments}, and \cite{rogers1980matrix} who also make use of this notation.} More explicitly this structure corresponds to the replacement of each component $F_{i,j}$ of $F$ by its respective row vector of derivatives, i.e.
\begin{equation}
{\partial F_{i,j} \over \partial x'} = \left({\partial F_{i,j} \over \partial x_1},\ldots,{\partial F_{i,j} \over \partial x_{n_x}} \right).
\end{equation}

When F is a vector, denoted by $f$ in the following, we can write the corresponding first-order differential:

\paragraph{The First-order Differential} Let $f(x)$ be a differentiable $s \times 1$ vector function of an $n_x \times 1$ vector of real variables $x$. Then we can use the notation of the derivative matrix to obtain
\begin{align*}
d f = \frac{\partial f}{\partial x}dx=f_{x}dx
\end{align*}
as the first differential of $f(\cdot)$ at $x$ with increment $dx$. In more detail,
\begin{align*}
d f(x) = \begin{bmatrix}d f_{1}(x)\\d f_{2}(x)\\\vdots\\d f_{s}(x)\end{bmatrix} &=
\begin{bmatrix}
\frac{\partial f_1(x)}{\partial x_{1}}dx_{1} + \frac{\partial f_1(x)}{\partial x_{2}}dx_{2} + \dots +\frac{\partial f_1(x)}{\partial x_{n_{x}}}dx_{n_{x}}\\
\frac{\partial f_2(x)}{\partial x_{1}}dx_{1} + \frac{\partial f_2(x)}{\partial x_{2}}dx_{2} + \dots +\frac{\partial f_2(x)}{\partial x_{n_{x}}}dx_{n_{x}}\\
\vdots \\
\frac{\partial f_{s}(x)}{\partial x_{1}}dx_{1} + \frac{\partial f_{s}(x)}{\partial x_{2}}dx_{2} + \dots +\frac{\partial f_{s}(x)}{\partial x_{n_{x}}}dx_{n_{x}}\\
\end{bmatrix}\\
&=
\begin{bmatrix}
\frac{\partial f_1(x)}{\partial x_{1}} & \frac{\partial f_1(x)}{\partial x_2} & \dots & \frac{\partial f_1(x)}{\partial x_{n_x}} \\
\frac{\partial f_2(x)}{\partial x_{1}} & \ddots								&				& \vdots								\\
\vdots								&											& \ddots&	\vdots								\\
\frac{\partial f_s(x)}{\partial x_{1}} & \dots								&\dots  & \frac{\partial f_s(x)}{\partial x_{n_x}}
\end{bmatrix}
\begin{bmatrix}
dx_{1}\\
dx_{2}\\
\vdots\\
dx_{n_{x}}
\end{bmatrix}\end{align*}
Using a more concise notation we get,
\begin{align*}
d f &=
\begin{bmatrix}
f_{1,x}\\
f_{2,x}\\
\vdots\\
f_{s,x}
\end{bmatrix}dx
= f_{x}d_{x}.
\end{align*}

\paragraph{The Second-order Differential} Let $f(x)$ be a twice differentiable $s \times 1$ vector function of an $n_{x} \times 1$ vector of real variables $x$. Then we can use the notation of the derivative matrix to obtain
\begin{align*}
d^2 f(x) = \frac{\partial \left(\frac{\partial f(x)}{\partial x}dx\right)}{\partial x}dx=\frac{\partial^2 f(x)}{\partial x^2}dx^{\otimes 2}=f_{x^2}dx^{\otimes 2}
\end{align*}
as the second differential of $f(\cdot)$ at $x$ with increment $dx$. In more detail,
\begin{align*}
d^2 f(x) & = d (d f(x)) = \begin{bmatrix} d (d f_{1}(x))\\d (d f_{2}(x))\\\vdots\\d (d f_{s}(x))\end{bmatrix}=
\begin{bmatrix}
d (f_{1,x}dx )\\
d (f_{2,x}dx )\\
\vdots\\
d (f_{s,x}dx )
\end{bmatrix}
=
\begin{bmatrix}
    \begin{bmatrix}df_{1,x_{1}} & df_{1,x_{2}} & \dots & df_{1,x_{n_{x}}}\end{bmatrix} dx\\
    \begin{bmatrix}df_{2,x_{1}} & df_{2,x_{2}} & \dots & df_{2,x_{n_{x}}}\end{bmatrix} dx\\
\vdots\\
    \begin{bmatrix}df_{s,x_{1}} & df_{s,x_{2}} & \dots & df_{s,x_{n_{x}}}\end{bmatrix} dx
\end{bmatrix}\\[5ex]
&=
\begin{bmatrix}
f_{1,x_{1},x}dx & f_{1,x_{2},x}dx & \dots & f_{1,x_{n_{x}},x}dx\\
f_{2,x_{1},x}dx & f_{2,x_{2},x}dx & \dots & f_{2,x_{n_{x}},x}dx\\
\vdots&\vdots&\ddots&\vdots\\
f_{s,x_{1},x}dx & f_{s,x_{2},x}dx & \dots & f_{s,x_{n_{x}},x}dx\\
\end{bmatrix}dx\\[5ex]
&=
\begin{bmatrix}
f_{x_{1},x}dx & f_{x_{2},x}dx & \dots & f_{x_{n_{x}},x}dx
\end{bmatrix}dx\\[5ex]
&=
\begin{bmatrix}
f_{x_{1},x} & f_{x_{2},x} & \dots & f_{x_{n_{x}},x}
\end{bmatrix}
\begin{bmatrix}
dx & \operatorname*{0}\limits_{n_{x}\times 1} & \dots & \operatorname*{0}\limits_{n_{x}\times 1}\\
\operatorname*{0}\limits_{n_{x}\times 1} & dx & \dots & \operatorname*{0}\limits_{n_{x}\times 1}\\
\vdots & \vdots & \ddots & \cdots\\
\operatorname*{0}\limits_{n_{x}\times 1}&\operatorname*{0}\limits_{n_{x}\times 1}& \dots& dx
\end{bmatrix}dx \\[5ex]
&=
f_{x,x}\left(I_{n_{x}}\otimes dx \right)dx
=
f_{x^2}\left(dx \otimes dx \right)=
f_{x^2}dx^{\otimes 2}.
\end{align*}

\paragraph{The $k$th-order Differential} Let $f(x)$ be a $k$ times differentiable $s \times 1$ vector function of an $n_{x} \times 1$ vector of real variables $x$. Then
\begin{align*}
d^k f(x) = \frac{\partial \left(\frac{\partial^{k-1} f(x)}{\partial x^{k-1}}dx^{\otimes k-1}\right)}{\partial x}dx=\frac{\partial^k f(x)}{\partial x^k}dx^{\otimes k}=f_{x^k}dx^{\otimes k}
\end{align*}
as the $k$-order differential of $f(\cdot)$ at $x$ with increment $dx$.  In more detail,
\begin{align*}
d^k f(x) & = d (d^{k-1} f(x)) = \begin{bmatrix} d (d^{k-1} f_{1}(x))\\d (d^{k-1} f_{2}(x))\\\vdots\\d (d^{k-1} f_{s}(x))\end{bmatrix}=
\begin{bmatrix}
d (f_{1,x^{k-1}}dx^{\otimes k-1} )\\
d (f_{2,x^{k-1}}dx^{\otimes k-1} )\\
\vdots\\
d (f_{s,x^{k-1}}dx^{\otimes k-1} )
\end{bmatrix}\\[5ex]
&=
\begin{bmatrix}
    \begin{bmatrix}df_{1,x_{1},x^{k-1}} & df_{1,x_{2},x^{k-1}} & \dots & df_{1,x_{n_{x},x^{k-1}}}\end{bmatrix} dx^{\otimes k-1}\\
    \begin{bmatrix}df_{2,x_{1},x^{k-1}} & df_{2,x_{2},x^{k-1}} & \dots & df_{2,x_{n_{x},x^{k-1}}}\end{bmatrix} dx^{\otimes k-1}\\
\vdots\\
    \begin{bmatrix}df_{s,x_{1},x^{k-1}} & df_{s,x_{2},x^{k-1}} & \dots & df_{s,x_{n_{x},x^{k-1}}}\end{bmatrix} dx^{\otimes k-1}
\end{bmatrix}\\[5ex]
&=
\begin{bmatrix}
    \begin{bmatrix}f_{1,x_{1},x^{k-1}}dx & f_{1,x_{2},x^{k-1}}dx & \dots & f_{1,x_{n_{x}},x^{k-1}}dx\end{bmatrix}\\
    \begin{bmatrix}f_{2,x_{1},x^{k-1}}dx & f_{2,x_{2},x^{k-1}}dx & \dots & f_{2,x_{n_{x}},x^{k-1}}dx\end{bmatrix}\\
\vdots\\
    \begin{bmatrix}f_{s,x_{1},x^{k-1}}dx & f_{s,x_{2},x^{k-1}}dx & \dots & f_{s,x_{n_{x}},x^{k-1}}dx\end{bmatrix}\\
\end{bmatrix}dx^{\otimes k-1}\\[5ex]
&=
\begin{bmatrix}
f_{x_{1},x^{k-1}}dx & f_{x_{2},x^{k-1}}dx & \dots & f_{x_{n_{x}},x^{k-1}}dx
\end{bmatrix}dx^{\otimes k-1}\\[5ex]
&=
\begin{bmatrix}
f_{x_{1},x^{k-1}} & f_{x_{2},x^{k-1}} & \dots & f_{x_{n_{x}},x^{k-1}}
\end{bmatrix}
\begin{bmatrix}
dx & \operatorname*{0}\limits_{n_{x}\times 1} & \dots & \operatorname*{0}\limits_{n_{x}\times 1}\\
\operatorname*{0}\limits_{n_{x}\times 1} & dx & \dots & \operatorname*{0}\limits_{n_{x}\times 1}\\
\vdots & \vdots & \ddots & \cdots\\
\operatorname*{0}\limits_{n_{x}\times 1}&\operatorname*{0}\limits_{n_{x}\times 1}& \dots& dx
\end{bmatrix}dx^{\otimes k-1} \\[5ex]
&=
f_{x,x^{k-1}}\left(I_{n_{x}^{k-1}}\otimes dx \right)dx^{\otimes k-1}
=
f_{x^k}\left(dx^{\otimes k-1} \otimes dx \right)=
f_{x^k}dx^{\otimes k}.
\end{align*}

Similarly, we can define a matrix differential by
\paragraph{Matrix Differential}
\begin{equation}
dF = {\partial F \over \partial x'}\left(I_t \otimes dx\right).
\end{equation}

We can now present rules of differentiation. We omit the addition of rule as it is identical to the univariate case.
\paragraph{Product Rule}
If the matrix product $FH$ exists, i.e. $t=p$, then
\begin{equation}
{\partial FH \over \partial x'} = {\partial F  \over \partial x'}\left(H \otimes I_{n_x}\right) + F {\partial H \over \partial x'}.
\end{equation}
\paragraph{Kronecker Product Rule}
\begin{align}
{\partial\left( F \otimes H\right) \over \partial x'} = & F \otimes {\partial H  \over \partial x'} + K_{s,p}\left( H \otimes {\partial F \over \partial x'}\right)\left(K_{q,t} \otimes I_{n_x}\right) \\
= & F \otimes {\partial H  \over \partial x'} + \left({\partial F \over \partial x'} \otimes H\right) \left(I_t \otimes K_{n_x,q} \right)
\end{align}
The second equality follows either from the application of the commutation of the matrices in the parentheses and some computation using commutation matrices, or by applying the extended product rule for Kronecker products below to the case of $n = 2$.

\paragraph{Extended Kronecker Product Rule}

\begin{align}
\label{eq:kronprodrule}
 {\partial \left(A_1 \otimes \ldots \otimes A_n\right) \over \partial x'} = & \left({\partial A_1 \over \partial x'} \otimes \bar{A}_{2:n}\right)\left(I_{s_1} \otimes K_{n_x,\prod_{j=2}^n s_j}\right) \\
 &\ldots + \sum_{k=2}^{n-1} \left(\bar{A}_{1:k-1}\otimes {\partial A_k \over \partial x'} \otimes \bar{A}_{k+1:n}\right) \left(I_{\prod_{j = 1}^k s_j} \otimes K_{n_x,\prod_{j=k+1}^n s_j}\right) \\
&\ldots + A_{1:n-1} \otimes {\partial A_n \over \partial x'}
\end{align}
with $A_{k:l} := A_k \otimes \cdots \otimes A_l$ and $A_i \in \mathcal{M}_{r_i \times s_i}$ for all $i$.

\begin{proof}
In the following $A_{i:j}$ will refer to the Kronecker product of all matrices $A_i$ though $A_j$ with $i < j$.
A typical element of $A_{1:n}$ is of the form
\begin{equation}
a^{(1)}_{i_1,j_1} \cdot a^{(2)}_{i_2,j_2}\cdot \ldots \cdot a^{(n)}_{i_n,j_n},
\end{equation}
where $\cdot$ is the multiplication operator acting on scalars and $A_k = \left(a^{(k)}_{i_k,j_k}\right)$.
The definition of matrix derivatives implies via
\begin{equation}
A_{1:n} \otimes {\partial \over \partial x'},
\end{equation}
that we can apply the product rule for scalars on each component individually. Therefore we get
\begin{align}
{\partial (a^{(1)}_{i_1,j_1} \cdot a^{(2)}_{i_2,j_2}\cdot \ldots \cdot a^{(n)}_{i_n,j_n}) \over \partial x'} &= \sum_{k=1}^n a^{(1)}_{i_1,j_1} \cdot \ldots \cdot {\partial a^{(k)}_{i_k,j_k} \over \partial x'} \cdot \ldots \cdot a^{(n)}_{i_n,j_n}.
\end{align}
This representation implies that we can compute the derivate of $A_{1:n}$ by computing $n$ individual term, where we keep all components constant, instead for the matrix we want to the take the derivative of. Let $\bar{A_{\cdot}}$ denote a matrix which is to be treated as a constant when taking the derivative.
\begin{equation}
{\partial \left(A_1 \otimes \ldots \otimes A_n\right) \over \partial x'} = {\partial(A_1 \otimes \bar{A}_{2:n}) \over \partial x'}+		 \sum_{k=2}^{n-1} {\partial\left(\bar{A}_{1:k-1}\otimes A_k \otimes \bar{A}_{k+1:n-1}\right) \over \partial x'} + {\partial(\bar{A}_{1:n-1} \otimes A_n) \over \partial x'}.
\end{equation}
Using that, and by applying commutation results from Section 3, we can write
\begin{align*}
{\partial \left(A_1 \otimes \ldots \otimes A_n\right) \over \partial x'}
=&	{\partial(A_1 \otimes \bar{A}_{2:n}) \over \partial x'}+	 \sum_{k=2}^{n-1} {\partial\left(\bar{A}_{1:k-1}\otimes A_k \otimes \bar{A}_{k+1:n}\right) \over \partial x'} + {\partial(\bar{A}_{1:n-1} \otimes A_n) \over \partial x'}.\\																																																														
=&	 \left(A_1 \otimes \bar{A}_{2:n}\right) \otimes {\partial \over \partial x'} + \sum_{k=2}^{n-1} \left(\bar{A}_{1:k-1}\otimes A_k \otimes \bar{A}_{k+1:n}\right) \otimes {\partial \over \partial x'}	 + \bar{A}_{1:n-1} \otimes {\partial A_n \over \partial x'}\\
 =& A_1 \otimes \left({\partial \over \partial x'}\otimes \bar{A}_{2:n}\right)K_{d,\prod_{j=2}^n s_j}  + \sum_{k=2}^{n-1}\bar{A}_{1:k-1}\otimes A_k \otimes \left({\partial \over \partial x'}		\otimes \bar{A}_{k+1:n}\right)K_{n_x,\prod_{j=k+1}^n s_j} \ldots \\
&  \ldots + A_{1:n-1} \otimes {\partial A_n \over \partial x'}\\
 =& A_1 I_{s_1} \otimes \left({\partial \over \partial x'}\otimes \bar{A}_{2:n}\right)K_{d,\prod_{j=2}^n s_j} \ldots \\
&\ldots + \sum_{k=2}^{n-1} \left(\bar{A}_{1:k-1}\otimes A_k\right) I_{\prod_{j = 1}^k s_j} \otimes \left({\partial \over \partial x'}		\otimes \bar{A}_{k+1:n}\right)K_{n_x,\prod_{j=k+1}^n s_j} \ldots \\
&  \ldots + A_{1:n-1} \otimes {\partial A_n \over \partial x'}\\
 =& \left({\partial A_1 \over \partial x'} \otimes \bar{A}_{2:n}\right)\left(I_{s_1} \otimes K_{n_x,\prod_{j=2}^n s_j}\right) \ldots \\
 &\ldots + \sum_{k=2}^{n-1} \left(\bar{A}_{1:k-1}\otimes {\partial A_k \over \partial x'} \otimes \bar{A}_{k+1:n}\right) \left(I_{\prod_{j = 1}^k s_j} \otimes K_{d,\prod_{j=k+1}^n s_j}\right) \ldots \\
&  + \ldots A_{1:n-1} \otimes {\partial A_n \over \partial x'}.
\end{align*}
\end{proof}
\paragraph{Chain Rule}
\begin{equation}
{\partial\left( F \circ G\right) \over \partial x'} = {\partial F \over \partial G'}\left(I_t \otimes {\partial G\over \partial x'}\right).
\end{equation}
If $F$ is a vector of size $s\times1$ this reduces to
\begin{equation}
{\partial\left( F \circ G\right) \over \partial x'} = {\partial F \over \partial G'}{\partial G\over \partial x'}.
\end{equation}

The differentiation rules above can also be expressed on their differential version:

\paragraph{Product Rule - Differential}
\begin{equation} \label{eq:diffproduct}
d FH = {\partial FH \over \partial x'}\left(I_q \otimes dx\right) = {\partial F  \over \partial x'}\left(H \otimes dx\right) + F {\partial H \over \partial x'}\left(I_q \otimes dx\right) .
\end{equation}

\paragraph{Product Rule - Kronecker Product - Differential}
\begin{equation}
d \left(F \otimes H\right) = {\partial\left( F \otimes H\right) \over \partial x'} \left(I_{tq} \otimes dx\right) = \left( F \otimes {\partial H  \over \partial x'} + \left({\partial F \over \partial x'} \otimes H\right) \left(I_t \otimes K_{n_x,q} \right)\right)\left(I_{tq} \otimes dx\right)
\end{equation}
\paragraph{Chain Rule - Differential}
\begin{align} \label{eq:diffchain}
d\left(F \circ G\right) &= {\partial\left( F \circ G\right) \over \partial x'}\left(I_t \otimes dx\right) \\
&= {\partial F \over \partial G'}\left(I_t \otimes {\partial G\over \partial x'}\right)\left(I_t \otimes dx\right) \\
&= {\partial F \over \partial G'}\left(I_t \otimes {\partial G\over \partial x'} dx\right).
\end{align}

\subsection*{B. Uniqueness and the Symmetrizer Matrix}

 For non-univariate function differentials the issue of non-uniqueness of the corresponding matrix or vector of partial derivatives arises.

 Given a matrix of partial derivatives of a vector function $f$ by $D_f$ we can write the k-th order differential using
 \begin{equation}
 d^k f = D_f \left(dx\right)^{\otimes k}.
 \end{equation}
 While the differential representation will be unique, the matrix of derivatives $D_f$ might not be.\footnote{See also \cite{chacon_higher_2021} for a general discussion of this issue.}

 To see this, note the following example.
 \begin{example}
 Consider the function
\begin{align}
f(x) &= {1 \over 2}a' x^{\otimes 2} \\
& = {1 \over 2} a_1 x_1^2 + \left(a_2 + a_3\right) x_1 x_2 +{1 \over 2} a_4 x_2^2
\end{align}
with $x=(x_1,x_2)'$ and $a = \left(a_1,a_2,a_3,a_4\right)'$. Writing down the second order differential explicitly provides
\begin{align}
d^2 f &=  a_1 d x_1^2 + \left(a_2+a_3 \right) dx_1 dx_2 + a_2 d x_2^2 \\
& = \left(a_1, {a_2+a_3\over 2}, {a_2+a_3\over 2}, a_4\right) (dx)^{\otimes 2}.
\end{align}
Using the second-order differential to identify the matrix of partial derivatives then leads us to write
\begin{equation}
{\partial^2 f \over \partial x'^{\otimes 2}} = \left(a_1, {a_2+a_3\over 2}, {a_2+a_3\over 2}, a_4\right).
\end{equation}
This is also exactly the same vector we would have gotten if would have taken partial derivatives of $f$ directly.
However, we can write the second order differential similarly by
\begin{align}
d^2 f &= \left(a_1, a_2, a_3, a_4\right) (dx)^{\otimes 2} \\
&=  a_1 d x_1^2 + \left(a_2+a_3 \right) dx_1 dx_2 + a_2 d x_2^2,
\end{align}
which would imply
\begin{equation}
{\partial^2 f \over \partial x'^{\otimes 2}} = \left(a_1, a_2, a_3, a_4\right).
\end{equation}
 \end{example}

We can make the decision to choose a unique matrix of derivatives by requiring that it fulfill some property of symmetry. Generally speaking, values of $D_f$ corresponding to the same component of the differential should be equal. For example, for a second order differential the values of $D_f$ that are multiplied with $dx_1 dx_2$ and $dx_2 dx_1$ should be identical. The symmetrizer matrix $S_{n,k}$ for vectors of size $n$ and Kronecker products with $k$ factors introduced by \cite{holmquist_direct_1985} achieves exactly this.
\begin{definition}
For the Kronecker product column vectors $x_1 \otimes x_2 \otimes \cdots \otimes x_m$ with $x_1,x_2,\ldots,x_m \in \R^n$, the symmetrizer matrix is given by the following property
\begin{equation}
S_{n,m} \left(x_1 \otimes x_2 \otimes \ldots \otimes x_m\right)  = {1 \over m!} \sum_{\sigma \in Sym(m)} x_{\sigma(1)} \otimes x_{\sigma(2)} \otimes \ldots \otimes x_{\sigma(m)}.
\end{equation}
\end{definition}

The symmetrizer matrix itself is symmetric. Therefore, for row vectors $x_1,x_2,\ldots,x_m \in \R^n$ we get
\begin{equation}
 \left(x_1 \otimes x_2 \otimes \ldots \otimes x_m\right) S_{n,m} = {1 \over m!} \sum_{\sigma \in Sym(m)} x_{\sigma(1)} \otimes x_{\sigma(2)} \otimes \ldots \otimes x_{\sigma(m)}.
\end{equation}
Using this definition, it then follows that
 \begin{equation}
 d^k f = D_f \left(dx\right)^{\otimes k}  = D_f S_{n,k} \left(dx\right)^{\otimes k},
 \end{equation}
 because $S_{n,k} \left(dx\right)^{\otimes k} = \left(dx\right)^{\otimes k}$. While the differential does not change using the transformation, the matrix of derivative is now given by $D_f S_{n,k}$.

\begin{example}
In our example above, we can derive the symmetrizer matrix in the following way:
For vectors $x, y$ of size $d \times 1$ the symmetrizer fulfills the following property
\begin{equation}
S_{d,2}\left(x \otimes y\right) = {x \otimes y + y \otimes x\over 2}
\end{equation}
for all vectors $x$ and $y$.
We know that there exists a commutation matrix $K_{d,d}$ such that $K_{d,d}\left(x \otimes y\right) = y \otimes x$.
Therefore we can write
\begin{align}
S_{d,2}\left(x \otimes y\right) &= I_{d^2}{x \otimes y \over 2} + K_{d,d}{x \otimes y\over 2} \\
																												&={1 \over 2}\left(I_{d^2} + K_{d,d}\right)\left(x \otimes y\right).
\end{align}
Since this holds for arbitrary choices of $x$ and $y$ it follows that
\begin{equation}
S_{d,2} = {1 \over 2}\left(I_{d^2} + K_{d,d}\right).
\end{equation}

For $d = 2$ this becomes
\begin{equation}
S_{2,2} = \begin{pmatrix}
1\; 0 \; 0 \; 0 \\
0\; {1 \over 2}\; {1 \over 2}\; 0\\
0\; {1 \over 2}\; {1 \over 2}\; 0\\
0\; 0\; 0\; 1
\end{pmatrix} = {1 \over 2}\left(I_2 + K_{2,2}\right).
\end{equation}
Thusly, we can write
\begin{align}
 																																				d^2 f			&= \left(a_1, a_2, s_3, a_4\right) S_{2,2}  (dx)^{\otimes 2} \\
                                                                 &= \left(a_1, {a_2+a_3\over 2}, {a_2+a_3\over 2}, a_4\right) S_{2,2}  (dx)^{\otimes 2}\\
                                                                 &={\partial^2 f \over \partial x'^{\otimes 2}}  (dx)^{\otimes 2}.
\end{align}

\end{example}

\paragraph{Symmetrizing Faà Di Bruno}
Similarly, for our generalization of Faà Di Bruno's formula, this implies that the differential representation
\begin{align}
d^{n} f(g(x))
= \sum_{k=1}^{n} f_{g^{k}}(g(x)) \textbf{B}^g_{n,k}  (dx)^{\otimes n},
\end{align}
implies a representation of the matrix of derivatives given by
\begin{equation}
{\partial^n (f \circ g) \over \partial (x')^{\otimes n}}  = \sum_{k=1}^{n} f_{g^{k}}(g(x)) \textbf{B}^g_{n,k} S_{n_x,n}.
\end{equation}
instead of
\begin{equation}
{\partial^n (f \circ g) \over \partial (x')^{\otimes n}} = \sum_{k=1}^{n} f_{g^{k}}(g(x)) \textbf{B}^g_{n,k} .
\end{equation}

\end{document}